\documentclass[A4paper, 12pt]{article}
\usepackage[utf8]{inputenc}
\usepackage[OT2, T1]{fontenc}
\usepackage{cite}
\usepackage{amsmath}
\usepackage{amsfonts}
\usepackage{amssymb}
\usepackage{amsxtra}
\usepackage{mathrsfs}
\usepackage{tensor}
\usepackage[thmmarks, amsmath, thref, amsthm]{ntheorem}
\usepackage[nottoc]{tocbibind}
\usepackage[pagebackref=true]{hyperref}

\usepackage{graphicx}
\usepackage[all]{xy}
\usepackage[portrait, top=2cm, bottom=2cm, left=2cm, right=2cm]{geometry}
\usepackage{enumerate}

\numberwithin{equation}{section}

\theoremstyle{plain}
\newtheorem{thm}{Theorem}[subsection]
\newtheorem{lemm}[thm]{Lemma}
\newtheorem{prop}[thm]{Proposition}
\newtheorem{coro}[thm]{Corollary}

\theoremstyle{definition}
\newtheorem{defn}[thm]{Definition}
\newtheorem{remk}[thm]{Remarks}

\newcommand{\alg}{\text{-}\operatorname{alg}}
\newcommand{\Aut}{\operatorname{Aut}}
\newcommand{\Bit}{\operatorname{Bit}}
\newcommand{\BIT}{\operatorname{BIT}}
\newcommand{\BPP}{\operatorname{BPP}}
\newcommand{\Br}{\operatorname{Br}}
\newcommand{\Coker}{\operatorname{Coker}}
\newcommand{\Ext}{\operatorname{Ext}}
\newcommand{\Gal}{\operatorname{Gal}}
\newcommand{\GL}{\operatorname{GL}}
\newcommand{\grp}{\text{-}\operatorname{grp}}
\newcommand{\Hom}{\operatorname{Hom}}
\newcommand{\id}{\operatorname{id}}
\newcommand{\iHom}{\mathscr{H}\kern -.5pt om}
\newcommand{\ind}{\operatorname{ind}}
\newcommand{\Inn}{\operatorname{Inn}}
\newcommand{\Ker}{\operatorname{Ker}}
\newcommand{\lcm}{\operatorname{lcm}}

\newcommand{\PGL}{\operatorname{PGL}}
\newcommand{\SL}{\operatorname{SL}}
\newcommand{\Spec}{\operatorname{Spec}}
\newcommand{\tors}{\operatorname{tors}}
\newcommand{\Mor}{\operatorname{Mor}}

\newcommand{\Cbb}{\mathbb{C}}
\newcommand{\Gbb}{\mathbb{G}}
\newcommand{\Hbb}{\mathbb{H}}
\newcommand{\Qbb}{\mathbb{Q}}
\newcommand{\Zbb}{\mathbb{Z}}

\newcommand{\Bcal}{\mathcal{B}}
\newcommand{\Dcal}{\mathcal{D}}
\newcommand{\Ecal}{\mathcal{E}}
\newcommand{\Fcal}{\mathcal{F}}
\newcommand{\Ocal}{\mathcal{O}}
\newcommand{\Pcal}{\mathcal{P}}
\newcommand{\Ucal}{\mathcal{U}}

\renewcommand{\C}{\mathrm{C}}
\renewcommand{\H}{\mathrm{H}}
\newcommand{\R}{\mathrm{R}}
\newcommand{\Z}{\mathrm{Z}}
\newcommand{\N}{\mathrm{N}}

\newcommand{\Cscr}{\mathscr{C}}
\newcommand{\Escr}{\mathscr{E}}
\newcommand{\Lscr}{\mathscr{L}}
\newcommand{\Vscr}{\mathscr{V}}
\newcommand{\Tscr}{\mathscr{T}}
\newcommand{\Xscr}{\mathscr{X}}

\DeclareSymbolFont{cyrletters}{OT2}{wncyr}{m}{n}
\DeclareMathSymbol{\Sha}{\mathalpha}{cyrletters}{"58}

\newcommand{\del}{\partial}

\title{Mayer--Vietoris sequences for complexes of tori}
\author{Nguy\~{\^{e}}n M\d{a}nh Linh}
\date{\today}

\begin{document}
\maketitle

{\em Corresponding author:} Nguy\~{\^{e}}n M\d{a}nh Linh, Institut de Math\'ematiques de Jussieu - Paris Rive Gauche, CNRS, Tour 15-16, Sorbonne Université, 4 place Jussieu, 75252 Paris cedex 05, France.

{\em Email address:} mlnguyen@imj-prg.fr

\begin{abstract}
    In the patching setting, given a factorization inverse system of fields over which patching for finite-dimensional vector spaces holds, together with a crossed module over the inverse limit field, the corresponding six-term Mayer--Vietoris sequence is constructed, generalizing the classical result of Harbater--Hartmann--Krashen for linear algebraic groups. When the crossed module is a two-term complex of tori, the above sequence is extended into a nine-term exact sequence, notably without any assumption on global domination of Galois cohomology of the inverse system. As an application, we show that patching holds for nonabelian second Galois cohomology of reductive groups with smooth centers. We then obtain a weak local--global principle for this cohomology set in the simply connected semisimple case. We also rediscover a well-known local--global principle for indices of central simple algebras.
\end{abstract}

{\em Keywords.} Galois cohomology, nonabelian cohomology, patching, linear algebraic groups.

{\em 2020 Mathematics subject classifications (MSC).} 11E72, 12E30.

\tableofcontents

\section{Introduction} \label{sec:Intro}

\subsection{Patching} \label{subsec:Context}

Given an algebraic number field $k$ and a finite group $G$, the classical {\em inverse Galois problem} asks whether there exists a finite Galois extension of $k$ with Galois group $G$. Despite having been solved for several families of finite groups, the problem remains open in full generality.

In recent years, interest in the inverse Galois problem over arithmetico-geometric analogues of algebraic number fields has been rising. Among the available techniques at our disposal, {\em patching} is a method that has yielded fruitful results. Inspired by the Riemann existence theorem over $\Cbb(X)$, the method was developed by Harbater, Hartmann, Krashen, and several other authors (see {\em e.g.} \cite{Harbater2003Patching, HH2010Patching}) to study arithmetic problems over {\em semiglobal fields}, that is, fields of the form $K(X)$ for some complete discretely valued field $K$ and some (smooth projective geometrically integral) $K$-curve $X$. A notable result in this direction is a theorem by Harbater, which states that any finite group is a Galois group over $\Qbb_p(t)$ \cite{Harbater1987Galois, Liu1995Galois, CT2000Galois, Kollar2000Fundamental, Kollar2003Rational}.

As in \cite[\S{2}]{HHK2015LocalGlobal}, we consider the general setting for patching over fields as follows. A {\em factorization inverse system of fields} is given by the data of a nonempty graph $\Gamma = (\Vscr,\Escr)$ on finitely many vertices and edges (multiple edges are allowed, but loops are not), together with a field $F_i$ for each vertex $i \in \Vscr$, a field $F_k$ for each edge $k \in \Escr$, and a field inclusion $F_i \hookrightarrow F_k$ whenever $i$ is a vertex of $k$. We assume that the inverse limit of the corresponding system (in the category of commutative rings) is a field $F$. In other words, we have an equalizer diagram
    \begin{equation*}
        1 \to F \to \prod_{i \in \Vscr} F_i \rightrightarrows \prod_{k \in \Vscr} F_k
    \end{equation*}
of fields (in particular, $\Gamma$ is connected). We refer to $\Gamma$ as the {\em reduction graph} of $\Fcal$.

The typical situation where patching is applied is the case where $F = K(X)$ is a semiglobal field as above. To construct the reduction graph $\Gamma$ as well as the overfields $F_i$ and $F_k$, we need some notations. Let $T$ be the valuation ring of $K$ and $t \in T$ a uniformizer. Let $\Xscr \to \Spec(T)$ be a flat, projective $T$-curve, which is normal and integral, with generic fibre $X$ (a {\em normal model} of $F$). The reduced closed fibre $Y \subseteq \Xscr$ is a union of finitely many smooth connected curves over the residue field $k$ of $T$. Fix a finite nonempty subset $\Pcal \subseteq Y$ containing the points at which $Y$ is not unibranched.

    \begin{itemize}
        \item For $P \in \Pcal$, denote by $F_P$ the field of fractions of the complete local ring $\hat{R}_P:=\hat{\Ocal}_{\Xscr,P}$ at $P$.

        \item Let $\Ucal$ denote the set of connected components of $Y \setminus \Pcal$. For $U \in \Ucal$, let $\hat{R}_U$ denote the $t$-adic completion of $R_U:=\bigcap_{P \in U} \Ocal_{\Xscr,P}$, and let $F_U$ be its field of fractions.
        
        \item Whenever a point $P \in \Pcal$ lies in the closure $\bar{U} \subseteq Y$ of a component $U \in \Ucal$, we obtain a {\em branch} $\wp = (U,P)$ at $P$, that is, a height one prime ideal of $\hat{R}_P$ containing $t$. Let $\hat{R}_{\wp}$ be the completion of the localization of $\hat{R}_P$ at $\wp$, and let $F_{\wp}$ denote its field of fractions. Let $\Bcal$ denote the set of branches at the points of $\Pcal$.
    \end{itemize}
The set $\Pcal$ then yields a (bipartite) graph $\Gamma_{\Pcal} = (\Ucal \sqcup \Pcal, \Bcal)$, and the family $((F_{\xi})_{\xi \in \Ucal \sqcup \Pcal}, (F_{\wp})_{\wp \in \Bcal})$ is a factorization inverse system of fields with limit $F$ (see \cite[Proposition 6.3]{HH2010Patching} and \cite[Proposition 3.3]{HHK2015LocalGlobal}).

In this article, for any factorization inverse system $\Fcal = ((F_i)_{i \in \Vscr}, (F_k)_{k \in \Escr})$ of fields, we assume that the reduction graph $\Gamma = (\Vscr,\Escr)$ is equipped with an orientation, that is, the choice of a head $l(k) \in \Vscr$ (hence, a tail $r(k) \in \Vscr$) for each edge $k \in \Escr$. We say that {\em patching for finite-dimensional vector spaces holds over $\Fcal$} if, given a finite-dimensional $F_i$-vector space $V_i$ for each $i \in \Vscr$ and an $F_k$-linear isomorphism $\mu_k: V_{l(k)} \otimes_{F_{l(k)}} F_k \xrightarrow{\cong} V_{r(k)} \otimes_{F_{r(k)}} F_k$ for each $k \in \Escr$, there exists a unique (up to isomorphism) $F$-vector space $V$ together with $F_i$-linear isomorphisms $\nu_i: V \otimes_F F_i \xrightarrow{\cong} V_i$ (for $i \in \Vscr$) such that $(\nu_{r(k)} \otimes_{F_{r(k)}} \id_{F_k}) = \mu_k \circ (\nu_{l(k)} \otimes_{F_{l(k)}} \id_{F_k})$ for all $k \in \Escr$ (note that this does not depend on the chosen orientation). In matrix form, this property is equivalent to the fact that for any integer $n \ge 1$ and any family $(g_k)_{k \in \Escr} \in \prod_{k \in \Escr} \GL_n(F_k)$, there exists a family $(g'_i)_{i \in \Vscr} \in \prod_{i \in \Vscr} \GL_n(F_i)$ such that $g_k = (g'_{r(k)})^{-1} g'_{l(k)}$ for all $k \in \Escr$ (the so-called {\em simultaneous matrix factorization property}) \cite[Proposition 2.2]{HHK2015LocalGlobal}. By \cite[Theorem 6.4]{HH2010Patching}, patching for finite-dimensional vector spaces holds over factorization inverse systems arising from semiglobal fields mentioned in the previous paragraph.

It is under the above conditions that many arithmetic problems over $F$ and over the fields $F_i$'s and $F_k$'s can be related. For instance, it is possible to show that patching for finite-dimensional vector spaces implies patching for {\em torsors} under linear algebraic groups. Recall that a torsor under a smooth $F$-algebraic group $G$ is a nonempty $F$-variety $P$ equipped with a right action $P \times_F G \to P$ which is simply transitive on $\bar{F}$-points (where $\bar{F}$ denotes a separable closure of $F$). These torsors are classified by the nonabelian Galois cohomology set $\H^1(F,G)$ (which can defined in terms of cocycles). The above result about patching torsors (for $G$ linear) can then be summarized in terms of a six-term {\em Mayer--Vietoris exact sequence}
    \begin{equation} \label{eq:SixTermMayerVietorisGroup}
        \xymatrix{
            1 \ar[r] & G(F) \ar[r] & \prod_{i \in \Vscr} G(F_i) \ar[r] & \prod_{k \in \Escr} G(F_k) \ar[lld]_{\delta} \\
            & \H^1(F,G) \ar[r] & \prod_{i \in \Vscr} \H^1(F_i,G) \ar@<-.5ex>[r] \ar@<.5ex>[r] & \prod_{k \in \Vscr} \H^1(F_k,G)
        }
    \end{equation}
of pointed sets. Here, the map $\prod_{i \in \Vscr} G(F_i) \to \prod_{k \in \Escr} G(F_k)$ is given by $(g'_i)_{i \in \Vscr} \mapsto ((g'_{r(k)})^{-1}g'_{l(k)})_{k \in \Vscr}$. The ``connecting map'' $\delta$ has the following description. Let $(g_k)_{k \in \Escr} \in \prod_{k \in \Escr} G(F_k)$. For $k \in \Escr$, let $\mu_k: G_{F_k} \to G_{F_k}$ be the left translation by $g_k$ (which is an isomorphism of $F_k$-torsors under $G$). Then $((G_{F_i})_{i \in \Vscr},(\mu_k)_{k \in \Escr})$ is a {\em $G$-torsor patching problem over $\Fcal$} ({\em cf.} \cite[p. 1566]{HHK2015LocalGlobal}), which has a unique (up to isomorphism) solution given by an $F$-torsor $P$ under $G$, given our assumption that patching for finite-dimensional vector spaces holds over $\Fcal$ \cite[Theorem 2.3]{HHK2015LocalGlobal}. Then, $\delta$ takes $(g_k)_{k \in \Escr}$ to the class $[P] \in \H^1(F,G)$ of this torsor ({\em cf.} \cite[Theorem 2.4]{HHK2015LocalGlobal} and its proof).

\subsection{Organization of the manuscript} \label{subsec:Organization}

The aim of the present article is to study analogues of \eqref{eq:SixTermMayerVietorisGroup} when the linear algebraic $F$-group $G$ is replaced by a $2$-term complex of $F$-tori. In Section \ref{sec:CrossedModule}, we consider the general nonabelian setting, where a {\em crossed module $\Cscr = [G \xrightarrow{\del} H]$} takes the place of the linear algebraic group $G$. It is possible to define the Galois hypercohomology groups $\Hbb^i(F, \Cscr)$ for $i = -1, 0$. While the definition of $\Hbb^{-1}(F,\Cscr)$ is rather straightforward, a geometric description of $\Hbb^0(F,\Cscr)$ can be given using {\em bitorsors} under the crossed module $\Cscr$ (see Definition \ref{defn:Bitorsor} and Proposition \ref{prop:BitorsorH0}). This slightly generalizes the classical notion of bitorsor under a linear algebraic group. With these descriptions at our disposal, we construct the six-term Mayer--Vietoris sequence for $\Cscr$ (Theorem \ref{thm:MayerVietorisCrossedModule}), which in particular entails the bitorsor patching problem. The proof relies on the Mayer--Vietoris sequence \eqref{eq:SixTermMayerVietorisGroup} for linear algebraic groups with some consideration on the ``bitorsor condition''. 

Next, we investigate the commutative setting. When $G$ is a commutative algebraic group, sequence \eqref{eq:SixTermMayerVietorisGroup} can be extended into a long exact sequence using the higher Galois cohomology groups $\H^i(-,G)$, under an additional condition of {\em global domination} of Galois cohomology \cite[Definition 2.3.3 and Theorem 2.5.1]{HHK2014LocalGlobal}. This latter holds for factorization inverse systems arising from semiglobal fields, with some restriction on the residual characteristic of $K$ \cite[Theorem 3.1.2]{HHK2014LocalGlobal}. In Section \ref{sec:MayerVietoris}, we consider the case where the crossed module $\Cscr$ is a $2$-term of complex $\Tscr = [T_1 \xrightarrow{\del} T_2]$ of $F$-tori (the Galois hypercohomology groups $\Hbb^i(-,\Tscr)$ are hence defined for all $i \ge -1$). There, the six-term Mayer--Vietoris sequence from Theorem \ref{thm:MayerVietorisCrossedModule} is extended into a nine-term exact sequence using $\Hbb^1(-,\Tscr)$ (Theorem \ref{thm:NineTermMayerVietorisComplexTori}). The proof is by d\'evissage and reducing to the case of quasitrivial tori. In particular, for any smooth $F$-group of multiplicative type $M$, we have a nine-term exact sequence relating $\H^i(-,M)$ for $i = 0,1,2$ (quite surprisingly, the validity of this sequence is {\em without any global domination condition} on Galois cohomology over $\Fcal$). 

Finally, in Section \ref{sec:Application}, we give an application of Theorem \ref{thm:NineTermMayerVietorisComplexTori} to the study of nonabelian second Galois cohomology of reductive groups. We prove that patching for (nonabelian) $\H^2$ holds for these groups whenever the center is smooth ({\em cf.} Corollary \ref{coro:PatchingNonabelianH2ConnectedReductive}); we insist on the fact that no global domination condition on the factorization inverse system is required. When the group is simply connected semisimple, we show a weak local--global principle for nonabelian $\H^2$ (Theorem \ref{thm:LocalGlobalPrincipleH2SimplyConnected}). To this end, we exploit the classification of center of simply connected groups and a theorem of Endo--Miyata on flasque tori split by a metacylic extension \cite{EM1975Tori}. Another application is an alternative proof of the local--global principle for indices of central simple algebras over a factorization inverse system ({\em cf.} Theorem \ref{thm:Index}). One can think of this result as a local--global principle for {\em neutrality} of elements in the nonabelian $\H^2$ set for groups of inner type ${}^1 A_n$ ({\em cf.} Remarks \ref{remk:Index}).

\subsection{Notations and conventions} \label{subsec:Notations}

The following notations shall be deployed throughout the article.

{\bf Fields.} When $F$ is a field, we use $\bar{F}$ to denote a fixed {\em separable} closure of $F$ and $\Gal_F:=\Gal_{\bar{F}/F}$ to denote its absolute Galois group. For an object $X$ defined over $F$ (variety, algebraic group...) and a field extension $E/F$, we write $X_E$ for the base change of $X$ to $E$). If $\eta$ is an element of some Galois hypercohomology set $\Hbb^r(F,-)$, its restriction to $E$ is an element of $\Hbb^r(E,-)$ and shall be denoted by $\eta|_E$.

{\bf Cohomology.} All cohomology groups and sets will be Galois cohomology. Unless specified, when we say that $[G \to H]$ is a $2$-term complex (commutative or noncommutative), we implicitly assume that $G$ and $H$ are placed in respective degrees $-1$ and $0$.

{\bf Algebraic groups.} When $G$ is an algebraic group over a field $F$, we use $Z(G)$ to denote its center. For a finite separable extension $E/F$, we write $\R_{E/F}$ for the restriction-of-scalars  functor {\em \`a la Weil}. By functoriality, it takes algebraic groups over $E$ to algebraic groups over $F$. The base change of this functor to $\bar{F}$ is $X \mapsto X \times_{\bar{F}} \cdots \times_{\bar{F}} X$ (${[E:F]}$ times) (if $E/F$ is Galois, then this is true already by base change to $E$).

{\bf Tate--Shafarevitch groups.} Let $\Fcal = ((F_i)_{i \in \Vscr}, (F_k)_{k \in \Escr})$ be a factorization inverse system of fields with inverse limit a field $F$. When $\Cscr$ is either a (bounded below) complex of commutative algebraic groups or a crossed module over $F$ (see Section \ref{sec:CrossedModule} below) and $r$ is an integer such that $\Hbb^r(E,\Cscr)$ is defined for any overfield $E/F$, we denote
    \begin{equation*}
        \Sha^r(F,\Cscr):=\Ker\left(\Hbb^r(F,\Cscr) \to \prod_{i \in \Vscr} \Hbb^r(F_i, \Cscr)\right).
    \end{equation*}

{\bf Cartier duality.} Let $F$ be a field. If $G$ is an algebraic $F$-group, we denote by $\hat{G} := \iHom_F(G,\Gbb_{m,F})$ its Galois module of geometric characters, that is, $\hat{G} = \Hom_{\bar{F}}(G_{\bar{F}},\Gbb_{m,\bar{F}})$ as a (finitely generated) abelian group, equipped with the action of $\Gal_F$ defined by $(\tensor[^s]{\chi}{})(g) := \tensor[^s]{\chi}{}(\tensor[^{s^{-1}}]{g}{})$ for all $s \in \Gal_F, \chi \in \hat{G}$, and $g \in G(\bar{F})$. An $F$-{\em group of multiplicative type} is an $F$-form of $\Gbb_{m,F}^n \times_F \prod_{i=1}^m \mu_{n_i,F}$, where $n \ge 0$ and $n_i \ge 1$. An $F$-{\em torus} is a connected $F$-group of multiplicative type ({\em i.e.}, an $F$-form of $\Gbb_{m,F}^n$). The formation $M \mapsto \hat{M}$ is an additive exact anti-equivalence between the category of smooth $F$-groups of multiplicative type and that of finitely generated abelian groups whose torsion part has order invertible in $F$, equipped with a continuous action of $\Gal_F$, called {\em Cartier duality}. We say that an $F$-group of multiplicative type $M$ is {\em split} by an extension $E/F$ if $\Gal_E$ acts trivially on $\hat{M}$. 

An $F$-torus $T$ is said to be {\em quasitrivial} if it is isomorphic to the product of tori of the form $\R_{E/F}(\Gbb_{m,E})$, where $E/F$ are finite separable extensions. In this case, one has $\H^1(F,T) = 0$ by Shapiro's lemma and Hilbert's Theorem 90.
 
\section{Crossed modules} \label{sec:CrossedModule}

\subsection{Galois hypercohomology and bitorsors} \label{subsec:CrossedModuleH0}

Let $F$ be a field. We recall the notion of crossed module from \cite[Definition 2.1]{Borovoi1992Hypercohomology}. A {\em crossed module} over $F$ is a complex $\Cscr = [G \xrightarrow{\del} H$] of smooth linear algebraic groups over $F$ (regarded as a $2$-term complex, where $G$ and $H$ are placed in respective degrees $-1$ and $0$) together with a left action 
    \begin{equation*}
        H \times_F G \to G, \quad (h,g) \mapsto h \cdot g
    \end{equation*}
of $H$ on $G$ (by automorphisms of algebraic groups) satisfying
    \begin{equation} \label{eq:CrossModule}
        \begin{cases}
            gg'g^{-1} & = \del(g) \cdot g'\\
            \del(h \cdot g) & = h \del(g) h^{-1}
        \end{cases}
    \end{equation}
for any $h \in H(\bar{F})$ and $g,g' \in G(\bar{F})$. Its $(-1)^{\text{st}}$ Galois hypercohomology group is defined by
    \begin{equation*}
        \Hbb^{-1}(F, \Cscr):=(\Ker \del)(F).
    \end{equation*}
The condition \eqref{eq:CrossModule} forces $\Ker \del$ to be central in $G$. In particular, the above group is abelian.

A $0$-cochain with coefficients in $\Cscr$ is, by definition, a pair $(\alpha,h)$, where $\alpha: \Gal_F \to G(\bar{F})$ is a continuous ({\em i.e.} locally constant) map and $h \in H(\bar{F})$. It is easily checked that the set $\C^0(F,\Cscr)$ of such cochains carries a natural group law, given by
    \begin{equation} \label{eq:CrossModuleGroupLaw}
        (\alpha,h) (\beta,h'):= (\tensor[^h]{\beta}{}\alpha , hh'),
    \end{equation}\
where $(\tensor[^h]{\beta}{}\alpha)_s:=(h \cdot \beta_s)\alpha_s $ for any $s \in \Gal_F$. 

Let $\Z^0(F,\Cscr)$ denote the subset of $\C^0(F,\Cscr)$ consisting of $0$-cochains $(\alpha,h)$ such that
    \begin{equation} \label{eq:CrossModuleCocycle}
        \begin{cases}
            \alpha_{st} & = \alpha_s \tensor[^s]{\alpha}{_t}\\
            \del(\alpha_s) \tensor[^s]{h}{} & = h
        \end{cases}
    \end{equation}
for any $s,t \in \Gal_F$; in particular, $\alpha$ is a Galois $1$-cocycle with coefficients in $G$. We then call $(\alpha,h)$ a $0$-cocycle with coefficients in $\Cscr$. One checks that $\Z^0(F,\Cscr)$ is a subgroup of $\C^0(F,\Cscr)$. The map
    \begin{equation*}
        G(\bar{F}) \to \C^0(F,\Cscr), \quad g \mapsto (\delta g,\del(g)),
    \end{equation*}
where $(\delta g)_s:= g \tensor[^s]{g}{^{-1}}$ for any $s \in \Gal_F$, is a group homomorphism whose image is a normal subgroup of $\Z^0(F,\Cscr)$, because
    \begin{equation*}
        (\alpha,h)(\delta g, \del(g)) = (\delta(h \cdot g), \del(h \cdot g)) (\alpha,h)
    \end{equation*}
for any $g \in G(\bar{F})$ and $(\alpha,h) \in \Z^0(F,\Cscr)$. We define $0^{\text{th}}$ Galois hypercohomology group $\Hbb^0(F, \Cscr)$ to be the quotient of $\Z^0(F,\Cscr)$ by this image. Therefore, two $0$-cocycles $(\alpha,h)$ and $(\beta,h')$ with coefficients in $\Cscr$ define the same cohomology class if and only if there exists $g \in G(\bar{F})$ such that
    \begin{equation} \label{eq:CrossModuleCoboundary}
        \begin{cases}
            h' & = \del(g) h\\
            \beta_s & = g \alpha_s \tensor[^s]{g}{^{-1}}
        \end{cases}
    \end{equation}
for any $s \in \Gal_F$.

\begin{defn} \label{defn:Bitorsor}
    An {\em $F$-bitorsor} under $\Cscr$ is a pair $(P,f)$, where $P$ is a right $F$-torsor under $G$ and $f: P \to H$ is a morphism of $F$-varieties such that $f(p \ast g) = f(p) \del(g)$ for any $p \in P(\bar{F})$ and $g \in G(\bar{F})$ (here, $\ast: P \times_F G \to P$ denotes the right action of $G$ on $P$). A morphism $(P,f) \to (P',f')$ of $F$-bitorsors is a $G$-equivariant morphism $\varphi: P \to P'$ such that $f' \circ \varphi = f$.
\end{defn}

Just like morphisms of torsors under $G$, any morphism of bitorsors under $\Cscr$ is {\em de facto} an isomorphism. In other words, the category $\BIT_F(\Cscr)$ of $F$-bitorsors under $\Cscr$ is a groupoid. Let $\Bit_{F}(\Cscr):=\pi_0(\BIT_F(\Cscr))$ denote the set of isomorphism classes of these bitorsors.

\begin{prop} \label{prop:BitorsorH0}
    There is a bijection between $\Hbb^0(F,\Cscr)$ and $\Bit_{F}(\Cscr)$, which is functorial with respect to the field $F$ and the crossed module $\Cscr$.
\end{prop}
\begin{proof}
    Given a $0$-cocycle $(\alpha,h) \in \Z^0(F,\Cscr)$, we have that $\alpha: \Gal_F \to G(\bar{F})$ is a cocycle in view of \eqref{eq:CrossModuleCocycle}. Let $P_\alpha$ be the $F$-torsor under $G$ it defines, that is, $P_{\alpha,\bar{F}} = G_{\bar{F}}$ equipped with the right translation action of $G_{\bar{F}}$ and the twisted Galois action
        \begin{equation*}
            \Gal_F \times G(\bar{F}) \to G(\bar{F}), \quad (s,p) \mapsto \alpha_s \tensor[^s]{p}{}. 
        \end{equation*}
    The $\bar{F}$-morphism
        \begin{equation*}
            P_{\alpha,\bar{F}} \to H_{\bar{F}}, 
            \quad p \mapsto h^{-1} \del(p) 
        \end{equation*}
    is compatible with the action of $\Gal_F$, because
        \begin{equation*}
            h^{-1} \del(\alpha_s \tensor[^s]{p}{}) = h^{-1} \del(\alpha_s) \del(\tensor[^s]{p}{}) = (\tensor[^s]{h}{^{-1}}) (\tensor[^s]{\del(p)}{}) = \tensor[^s]{(h^{-1}\del(p))}{}
        \end{equation*}
    for any $s \in \Gal_F$, in view of \eqref{eq:CrossModuleCocycle}. By Galois descent, it gives rise to a morphism $f_{\alpha,h}: P_\alpha \to H$ of $F$-varieties. By construction, it is easy to see that $(P_{\alpha},f_{\alpha,h})$ is an $F$-bitorsor under $\Cscr$. 
    
    Let $(\beta,h') \in \Z^0(F,\Cscr)$ be another $0$-cocycle. It is well-known that any $G$-equivariant isomorphism $\varphi: P_\alpha \to P_\beta$ is given by left translation by an element $g \in G(\bar{F})$ satisfying the second identity in \eqref{eq:CrossModuleCoboundary}. One also checks that $\varphi$ is an isomorphism $(P_{\alpha},f_{\alpha,h}) \to (P_{\beta}, f_{\beta,h'})$ of bitorsors precisely when the first identity in \eqref{eq:CrossModuleCoboundary} holds true. It follows that we have an injective map
        \begin{align*}
            \Phi: \Hbb^0(F,\Cscr) & \to \Bit_{F}(\Cscr) \\
            [\alpha,h] & \mapsto [P_\alpha,f_{\alpha,h}].
        \end{align*}
    We show the surjectivity of $\Phi$. Let $(P,f)$ be a bitorsor under $\Cscr$. Choose a point $p \in P(\bar{F})$ and define the $1$-cocycle  $\alpha: \Gal_F \to G(\bar{F})$ via the relation
        \begin{equation*}
            \tensor[^s]{p}{} = p \ast \alpha_s
        \end{equation*}
    for $s \in \Gal_F$. Then
        \begin{equation*}
            \varphi: P_\alpha \to P, \quad g \mapsto p \ast g
        \end{equation*}
    is an isomorphism of $F$-torsors under $G$. Let $h:=f(p)^{-1} \in H(\bar{F})$. For $s \in \Gal_F$, one has
        \begin{equation*}
            \tensor[^s]{h}{^{-1}} = \tensor[^s]{f(p)}{} = f(\tensor[^s]{p}{}) = f(p \ast \alpha_s) = f(p) \del(\alpha_s) = h^{-1} \del(\alpha_s)
        \end{equation*}
    by definition of bitorsors. In view of \eqref{eq:CrossModuleCocycle}, the pair $(\alpha,h)$ is a $0$-cocycle with coefficients in $\Cscr$. Furthermore, one has $f \circ \varphi = f_{\alpha,h}$, for if $g \in P_\alpha(\bar{F})$, then
        \begin{equation*}
            f(\varphi(g)) = f(p \ast g) = f(p) \del(g) = h^{-1} \del(g) = f_{\alpha,h}(g).
        \end{equation*}
    Thus, we have $[P,f] = [P_{\alpha},f_{\alpha,h}]$, proving that $\Phi$ is surjective. Functoriality follows easily from construction.
\end{proof}

It is not hard to see that the neutral element $[(s \mapsto 1_G),1_H]$ of $\Hbb^0(F,\Cscr)$, under the bijection from Proposition \ref{prop:BitorsorH0}, corresponds to the class of the {\em trivial bitorsor} $(G,\del)$.

\begin{remk} \label{remk:BitorsorGroupLaw}
    If $(P,f)$ is an $F$-bitorsor under $\Cscr$, then the $F$-morphism
        \begin{equation*}
            \lambda:G \times_F P \to P, \quad (g,p) \mapsto p \ast (f(p)^{-1} \cdot g)
        \end{equation*}
    is a left action of $G$ on $P$. Indeed, for $p \in P(\bar{F})$, one has 
        \begin{equation*}
            \lambda(1_G,p) = p \ast (f(p)^{-1} \cdot 1_G) = p \ast 1_G = p.
        \end{equation*}
    Furthermore, for $g,g' \in G(\bar{F})$, since   
        \begin{equation} \label{eq:BitorsorLeftAction}
            f(\lambda(g,p)) = f(p \ast (f(p)^{-1} \cdot g)) = f(p) \del(f(p)^{-1} \cdot g) =  f(p) f(p)^{-1} \del(g) f(p) = \del(g) f(p).
        \end{equation}
    in view of \eqref{eq:CrossModule}, one has
        \begin{align*}
            \lambda(g',\lambda(g,p)) & = \lambda(g,p) \ast (f(\lambda(g,p))^{-1} \cdot g') \\
            & = (p \ast (f(p)^{-1} \cdot g)) \ast ((f(p)^{-1}\del(g)^{-1}) \cdot g') \\
            & = (p \ast (f(p)^{-1} \cdot g)) \ast (f(p)^{-1} \cdot (g^{-1} g' g)), & \text{by \eqref{eq:CrossModule},} \\
            & = p\ast ((f(p)^{-1} \cdot g)(f(p)^{-1} \cdot (g^{-1} g' g))) \\
            & = p \ast (f(p)^{-1} \cdot (gg^{-1} g' g)) \\
            & = p \ast (f(p)^{-1} \cdot (g'g)) \\
            & = \lambda(g'g, p).
        \end{align*}
    This left action is compatible with the right action $\ast: P \times_F G \to P$. Indeed, one has
        \begin{align*}
            \lambda(g, p \ast g') & = p \ast g' \ast (f(p \ast g')^{-1} \cdot g) \\
            & = (p \ast g') \ast ((\del(g')^{-1} f(p)^{-1}) \cdot g) \\
            & = (p \ast g') \ast ((g')^{-1}(f(p)^{-1} \cdot g)g'), & \text{by \eqref{eq:CrossModule},} \\
            & = (p \ast (f(p)^{-1} \cdot g)) \ast g'\\
            & = \lambda(g,p) \ast g'.
        \end{align*}
    In other words, the actions $\lambda$ and $\ast$ equip $P$ with a structure of an $F$-bitorsor under $G$ in the traditional sense (this justifies the terminology).

    Let $(P',f')$ be a second $F$-bitorsor under $\Cscr$. One may form the {\em contracted product} $P'':=P' \times^G_F P$, that is, $P''$ is the quotient of $P' \times_F P$ by the left action of $G$ given by
        \begin{equation} \label{eq:BitorsorContractedProduct}
            G \times_F P' \times_F P \to P' \times_F P, \quad (g,p',p) \mapsto (p' \ast g^{-1}, \lambda(g,p)).
        \end{equation}
    This is again an $F$-bitorsor under $G$. In particular, if $\pi: P' \times_F P \to P''$ denotes the canonical projection, then the right action $\ast: P'' \times_F G \to P''$ of $G$ on $P''$ satisfies
        \begin{equation*}
            \pi(p',p) \ast g = \pi(p', p \ast g)
        \end{equation*}
    for all $p' \in P'(\bar{F})$, $p \in P(\bar{F})$, and $g \in G(\bar{F})$. Furthermore, the $F$-morphism
        \begin{equation*}
            P' \times_F P \to H, \quad (p',p) \mapsto f'(p')f(p)
        \end{equation*}
    is invariant with respect to the left action \eqref{eq:BitorsorContractedProduct}. Indeed, one has
        \begin{equation*}
            f'(p' \ast g^{-1}) f(\lambda(g,p)) = f'(p') \del(g^{-1}) \del(g) f(p) = f'(p')f(p)
        \end{equation*}
    by virtue of \eqref{eq:BitorsorLeftAction} and the very definition of bitorsors under a crossed module. Therefore, the above morphism induces an $F$-morphism $f'': P'' \to H$ such that $f''(\pi(p',p)) = f'(p')f(p)$ for all $p' \in P'(\bar{F})$ and $p \in P(\bar{F})$. Additionally, it is clear that the pair $(P'',f'')$ is an $F$-bitorsor under $\Cscr$.

    By Proposition \ref{prop:BitorsorH0}, the set $\Bit_F(\Cscr)$ is equipped with a natural group law. We claim that the class $[P'',f'']$ is precisely the product $[P,f] [P',f']$. Indeed, following the proof of the same proposition, the choice of geometric points $p \in P(\bar{F})$ and $p' \in P'(\bar{F})$ yields $0$-cocycles $(\alpha,h)$ and $(\beta,h')$ representing $[P,f]$ and $[P',f']$ respectively, such that
        \begin{equation*}
            \tensor[^s]{p}{} = p \ast \alpha_s, \quad h = f(p)^{-1}, \quad \tensor[^s]{p}{}' = p' \ast \beta_s, \text{ and} \quad h' = f(p')^{-1}
        \end{equation*}
    for any $s \in \Gal_F$. Consider the point $p'':=\pi(p',p) \in P''(\bar{F})$. It satisfies $f''(\pi(p',p))^{-1} = f(p)^{-1}f'(p')^{-1} = hh'$ and
        \begin{align*}
            \tensor[^s]{\pi(p',p)}{} & = \pi(\tensor[^s]{p}{}',\tensor[^s]{p}{}) = \pi(p' \ast \beta_s, p \ast \alpha_s) = \pi(p' \ast \beta_s, p) \ast \alpha_s = \pi(p', \lambda(\beta_s,p)) \ast \alpha_s \\
            & = \pi(p', p \ast (f(p)^{-1} \cdot \beta_s)) \ast \alpha_s = \pi(p', p \ast(h \cdot \beta_s)) \ast \alpha_s = \pi(p',p) \ast ((h \cdot \beta_s) \alpha_s).
        \end{align*}
    In view of \eqref{eq:CrossModuleGroupLaw}, the class $[P'',f'']$ is represented by the product cocycle $(\alpha,h)(\beta,h')$.
\end{remk}

\begin{lemm} \label{lemm:TrivialBitorsor}
    Let $(P,f)$ be an $F$-bitorsor under $\Cscr$. Then, the following are true.

    \begin{enumerate}
        \item \label{lemm:TrivialBitorsor:1} The bitorsor $(P,f)$ is trivial (that is, isomorphic to $(G,\del)$) if and only if the fibre $f^{-1}(1_H)$ contains an $F$-rational point (in particular, it is nonempty).

        \item \label{lemm:TrivialBitorsor:2} The automorphism group of the trivial bitorsor $(G,\del)$ is  $\Hbb^{-1}(F,\Cscr) = (\Ker \del)(F)$.
    \end{enumerate}
\end{lemm}
\begin{proof}
    Clearly, the fibre $\del^{-1}(1_H)$ contains the $F$-rational point $1_G \in G(F)$. Conversely, if $p \in P(F)$ such that $f(p) = 1_H$, then the $G$-equivariant isomorphism 
        \begin{equation*}
            \varphi: G \to P, \quad g \mapsto p \ast g
        \end{equation*}
    satisfies $f(\varphi(g)) = f(p \ast g) = f(p) \del(g) = \del(g)$ for all $g \in G(\bar{F})$, that is, $f \circ \varphi = \del$. Thus, $\varphi$ yields an isomorphism of $F$-bitorsors $(G,\del) \xrightarrow{\cong} (P,f)$. This proves \ref{lemm:TrivialBitorsor:1}.

    It is well-known that any automorphism of the trivial torsor $G$ is given by left translation by a unique element $g \in G(F)$. This is an automorphism of $(G,\del)$ if and only if $\del(gh) = \del(h)$ for any $h \in G(\bar{F})$, or $\del(g) = 1_H$, that is, $g \in (\Ker \del)(F)$. This proves \ref{lemm:TrivialBitorsor:2}.
\end{proof}

\begin{remk} \label{remk:FromTorsorToBitorsor}
    The inclusion $\Ker \del \hookrightarrow G$ induces a natural morphism 
        \begin{equation*}
            \iota: (\Ker \del)[1] = [\Ker \del \to 1] \to \Cscr
        \end{equation*}
    of crossed modules over $F$. It yields a group homomorphism
        \begin{equation*}
            \iota_\ast: \H^1(F,\Ker \del) \to \Hbb^0(F,\Cscr).
        \end{equation*}
    Algebrically, it takes the class of a $1$-cocycle $\alpha: \Gal_F \to \Ker \del$ to that of the $0$-cycle $(\alpha,1_H)$. Geometrically, it takes the class of an $F$-torsor $P$ under $\Ker \partial$ to the class of the $F$-bitorsor $(P \times_F^{\Ker \del} G, f)$ under $\Cscr$. Here, $P \times_F^{\Ker \del} G$ denotes the usual contracted product construction, on which $G$ acts on the second factor via right translation, and $f$ is induced by the morphism $\del: G \to H$ on the second factor. In particular, one easily checks, using Lemma \ref{lemm:TrivialBitorsor}, that $\iota_\ast$ has trivial kernel, hence injective. Nevertheless, it need not be surjective. Indeed, let $F$ be an algebraically closed field, $H$ a reductive but not semisimple connected linear algebraic $F$-group, and $G$ the universal cover of the semisimple group $[H,H]$. The natural composite
        \begin{equation*}
            G \twoheadrightarrow [H,H] \hookrightarrow H,
        \end{equation*}
    denote by $\del$, is then nonsurjective. Note that $H$ and $G$ have the same adjoint group, so letting $H$ act on $G$ by inner automorphisms yield a crossed module $\Cscr = [G \xrightarrow{\del} H]$. With these choices, we have $\H^1(F,\Ker \del) = 0$ (because $F$ is algebraically closed), but $\Hbb^0(F,\Cscr) = \Coker \del$ is the group of $F$-points of the nontrivial torus $H/[H,H]$. Geometrically, the nonsurjectivity of $\iota_\ast$ corresponds to the fact that if $(P,f)$ is an $F$-bitorsor under $\Cscr$, then the fibre $f^{-1}(1_H)$ need not be an $F$-torsor under $\Ker\del$. In fact, it is so if and only if the fibre  is {\em nonempty}, that is, contains an $\bar{F}$-point.
\end{remk}

\subsection{The six-term exact sequence}

In this subsection, let $\Fcal = ((F_i)_{i \in \Vscr}, (F_k)_{k \in \Escr})$ be a factorization inverse system of fields, whose inverse limit is a field $F$. When patching holds for finite-dimensional vector spaces over $\Fcal$, we propose the following analogue of \eqref{eq:SixTermMayerVietorisGroup} for crossed modules (which also slightly generalizes the Mayer--Vietoris sequence for bitorsors in the classical sense, established by Haase in \cite[Theorem 6.2.5]{Haase2018Thesis}).

\begin{thm} \label{thm:MayerVietorisCrossedModule}
    Let $\Cscr = [G \xrightarrow{\del} H]$ be a crossed module over $F$, with $G$ and $H$ linear. If patching holds for finite-dimensional vector spaces over $\Fcal$, then we have a functorial exact sequence
        \begin{equation} \label{eq:MayerVietorisCrossedModule}
        \xymatrix{
            1 \ar[r] & \Hbb^{-1}(F,\Cscr) \ar[r] & \prod_{i \in \Vscr} \Hbb^{-1}(F_i,\Cscr)  \ar[r] & \prod_{k \in \Escr} \Hbb^{-1}(F_k,\Cscr)  \ar[lld]_{\delta} \\
            & \Hbb^0(F,\Cscr) \ar[r] & \prod_{i \in \Vscr} \Hbb^0(F_i,\Cscr) \ar[r] & \prod_{k \in \Escr} \Hbb^0(F_k,\Cscr).
        }
    \end{equation}
    of pointed sets, where all arrows, except possibly the last one, are group homomorphisms.
\end{thm}

The claim in Theorem \ref{thm:MayerVietorisCrossedModule} for the top row is easy. Indeed, since $\Ker \del$ is a linear ($=$ affine) algebraic group, we have an exact sequence
    \begin{equation*}
        1 \to (\Ker \del)(F) \to \prod_{i \in \Vscr}(\Ker \del)(F_i) \to \prod_{k \in \Escr}(\Ker \del)(F_k),
    \end{equation*}
where the last arrow is given by $(g_i)_{i \in \Vscr} \mapsto (g_{r(k)}^{-1} g_{l(k)})_{k \in \Escr}$. Note that the latter is a homomorphism because $\Ker \del$ is commutative (it is central in $G$). This establishes the top row of \eqref{eq:MayerVietorisCrossedModule}, by definition of $\Hbb^{-1}$. As for the bottom row, in view of Proposition \ref{prop:BitorsorH0}, we are to establish an exact sequence
    \begin{equation} \label{eq:PatchingBitorsor}
        \Bit_F(\Cscr) \to \prod_{i \in \Vscr} \Bit_{F_i}(\Cscr) \rightrightarrows \prod_{k \in \Escr} \Bit_{F_k}(\Cscr),
    \end{equation}
that is, to show that {\em patching holds for $\Cscr$-bitorsors over $\Fcal$} ({\em cf.} \cite[Theorem 6.2.2]{Haase2018Thesis}). To this end, it is convenient to introduce the following notion.

\begin{defn} \label{defn:PatchingProblemBitorsor}
    Let $\Cscr = [G \xrightarrow{\del} H]$ be a crossed module over $F$. The groupoid $\BPP_\Fcal(\Cscr)$ of {\em $\Cscr$-bitorsor patching problems over $\Fcal$} is the groupoid where

    \begin{enumerate}
        \item an object is a collection $((P_i,f_i)_{i \in \Vscr}, (\mu_k)_{k \in \Escr})$, where $(P_i,f_i)$ is an $F_i$-bitorsor under $\Cscr$ for each $i \in \Vscr$, and $\mu_k: (P_{l(k),F_k}, f_{l(k),F_k}) \xrightarrow{\cong} (P_{r(k),F_k}, f_{r(k),F_k})$ is an isomorphism of $F_k$-bitorsors under $\Cscr$,

        \item an isomorphism $((P_i,f_i)_{i \in \Vscr}, (\mu_k)_{k \in \Escr}) \to ((P'_i,f'_i)_{i \in \Vscr}, (\mu'_k)_{k \in \Escr})$ is a collection $(\nu_i)_{i \in \Vscr}$, where $\nu_i: (P_i,f_i) \xrightarrow{\cong} (P'_i,f'_i)$ is an isomorphism of $F_i$-bitorsors under $\Cscr$ for each $i \in \Vscr$, such that $\mu_k' \circ \nu_{l(k),F_k} = \nu_{r(k),F_k} \circ \mu_k$ for any $k \in \Escr$,

        \item composition of isomorphisms and identities are defined component-wise.
    \end{enumerate}
\end{defn}

\begin{prop} \label{prop:PatchingBitorsor}
    Let $\Cscr = [G \xrightarrow{\del} H]$ be a crossed module over $F$, with $G$ and $H$ linear. If patching holds for finite-dimensional vector spaces over $\Fcal$, then the base change functor
    \begin{equation*}
        \Phi: \BIT_F(\Cscr) \to \BPP_\Fcal(\Cscr)
    \end{equation*}
    given by $\Phi(P,f) = ((P_{F_i}, f_{F_i})_{i \in \Vscr}, (\id_{P_{F_k}})_{k \in \Escr})$ on objects, and $\Phi(\mu) = (\mu_{F_i})_{i \in \Vscr}$ on (iso-)morphisms, is an equivalence of categories. In particular, every patching problem $((P_i, f_i)_{i \in \Vscr}, (\mu_k)_{k \in \Escr}) \in \BIT_F(\Cscr)$ has a unique (up to isomorphism) {\em solution}, that is, an $F$-bitorsor $(P,f)$ under $\Cscr$ such that there exists an isomorphism
        \begin{equation*}
            (\nu_i)_{i \in \Vscr}: ((P_{F_i}, f_{F_i})_{i \in \Vscr}, (\id_{P_{F_k}})_{k \in \Escr}) \xrightarrow{\cong} ((P_i, f_i)_{i \in \Vscr}, (\mu_k)_{k \in \Escr})
        \end{equation*}
    in $\BPP_\Fcal(\Cscr)$.
\end{prop}
\begin{proof}
    We need the following lemma on ``patching morphisms to an affine scheme''.
\begin{lemm} \label{lemm:PatchingMorphism}
    Let $X = \Spec(A)$ be an affine $F$-scheme and $S$ an $F$-scheme. Then 
        \begin{equation*}
            1 \to X(S) \to \prod_{i \in \Vscr} X(S_{F_i}) \rightrightarrows \prod_{k \in \Escr} X(S_{F_k})
        \end{equation*}
    is an equalizer diagram.
\end{lemm}
\begin{proof}
    The claim in the lemma is equivalent to saying that
        \begin{equation} \label{eq:PatchingAlgebra}
            0 \to \Hom_{F\alg}(A, B) \to \prod_{i \in \Vscr} \Hom_{F\alg}(A, B \otimes_F F_i) \rightrightarrows \prod_{k \in \Escr} \Hom_{F\alg}(A, B \otimes_F F_k),
        \end{equation}
    where $B:=\mathcal{O}_S(S)$, is an equalizer diagram. To see this, observe that, since 
        \begin{equation*}
            0 \to F \to \prod_{i \in \Vscr} F_i \rightrightarrows \prod_{k \in \Escr} F_k
        \end{equation*}
    is an equalizer diagram in the abelian category of $F$-vector spaces, successively applying the left exact functors $B \otimes_F -$ and $\Hom_F(A,-)$ yields an equalizer diagram
        \begin{equation*}
            0 \to \Hom_{F}(A, B) \to \prod_{i \in \Vscr} \Hom_{F}(A, B \otimes_F F_i) \rightrightarrows \prod_{k \in \Escr} \Hom_{F}(A, B \otimes_F F_k).
        \end{equation*}
    This already shows the injectivity of the first nontrivial arrow ({\em i.e.} exactness at the first nontrivial term) in \eqref{eq:PatchingAlgebra}. As for exactness at the second nontrivial term, assume that $\varphi: A \to B$ is an $F$-linear map such that the composite $A \xrightarrow{\varphi} B \hookrightarrow B \otimes_F F_i$ is an $F$-algebra homomorphism for all $i \in \Vscr$. Since $\Vscr$ is nonempty, we can take any $i \in \Vscr$, and the injectivity of $B \hookrightarrow B \otimes_F F_i$ shows that $\varphi$ is itself a homomorphism of $F$-algebras. The proof concludes.
\end{proof}
    
    Back to the proof of Proposition \ref{prop:PatchingBitorsor}. Let $(P,f)$ and $(P',f')$ be $F$-bitorsors under $\Cscr$. Assume that $\nu,\nu': (P,f) \xrightarrow{\cong} (P',f')$ are isomorphisms of $F$-bitorsors under $\Cscr$ such that $\nu_{F_i} = \nu_{F_i'}$ for all $i \in \Vscr$. Since $P'$ is a torsor under $G$, it is an affine $F$-scheme, so Lemma \ref{lemm:PatchingMorphism} ensures that $\nu = \nu'$. Therefore, $\Phi$ is faithful. 
    
    Furthermore, if $\nu_i: (P_{F_i},f_{F_i}) \xrightarrow{\cong} (P_{F_i}',f_{F_i}')$ is an isomorphism of $F_i$-bitorsors under $\Cscr$ for each $i \in \Vscr$, such that $\nu_{l(k),F_k} = \nu_{r(k),F_k}$ for all $k \in \Escr$, then (again by Lemma \ref{lemm:PatchingMorphism}) there is a morphism $\nu: P \to P'$ (of $F$-varieties) such that $\nu_{F_i} = \nu_i$ for all $i \in \Vscr$. Take any $i \in \Vscr$ and fix a field inclusion $\bar{F} \hookrightarrow \bar{F}_i$  extending $F \hookrightarrow F_i$. Then $P(\bar{F})$ can be regarded as a subset of $P(\bar{F}_i)$. Since $\nu_{F_i} = \nu_i$ is an isomorphism of $F_i$-torsors under $G$ and $f'_{F_i} \circ \nu_{i} = f_{F_i}$, one sees that $\nu$ is an isomorphism of $F$-torsors under $G$ and $f' \circ \nu = f$ by checking these conditions on $\bar{F}$-points. Therefore, $\Phi$ is full.    

    Let us now show that $\Phi$ is essentially surjective. Let $((P_i, f_i)_{i \in \Vscr}, (\mu_k)_{k \in \Escr})$  be an object of $\BPP_\Fcal(\Cscr)$. Since $G$ is linear, \cite[Theorem 2.3]{HHK2015LocalGlobal} ensures the existence of an $F$-torsor $P$ under $G$ together with isomorphisms
        \begin{equation*}
            \nu_i: P_{F_i} \xrightarrow{\cong} P_{i}
        \end{equation*}
    of $F_i$-torsors under $G$, for $i \in \Vscr$, such that $\mu_k \circ \nu_{l(k),F_k} = \nu_{r(k),F_k}$ for any $k \in \Escr$. It remains to construct a morphism $f: P \to H$ of $F$-varieties such that $(P,f)$ is an $F$-bitorsor under $\Cscr$ and $f_{i} \circ \nu_i = f_{F_i}$ for all $i \in \Vscr$.
    Since $\mu_k$ is furthermore an isomorphism of $F_k$-bitorsors, one has
        \begin{equation*}
            f_{r(k),F_k} \circ \nu_{r(k),F_k} = f_{r(k),F_k} \circ \mu_k \circ \nu_{l(k),F_k} = f_{l(k),F_k} \circ \nu_{l(k),F_k}.
        \end{equation*}
    By Lemma \ref{lemm:PatchingMorphism}, the morphisms $f_i \circ \nu_i$ glue into a morphism $f: P \to H$. Take any $i \in \Vscr$ and let $\bar{F} \hookrightarrow \bar{F}_i$ be a field inclusion extending $F \hookrightarrow F_i$. Then, for any $p \in P(\bar{F})$ and $g \in G(\bar{F})$, one has
        \begin{equation*}
            f(p \ast g) = f_i(\nu_i(p \ast g)) = f_i(\nu_i(p) \ast g) = f_i(\nu_i(p)) \del(g) = f(p) \del(g),
        \end{equation*}
    where $\ast$ denotes the right action of $G$ (resp. $G_{F_i}$) on $P$ (resp. $P_i$). Therefore, $(P,f)$ is indeed a bitorsor under $\Cscr$, concluding the proof of the proposition.
\end{proof}

\begin{proof} [Proof of Theorem \ref{thm:MayerVietorisCrossedModule}]
    We have already established the first row of \eqref{eq:MayerVietorisCrossedModule}. The second row is obtained from \eqref{eq:PatchingBitorsor} (which is established in Proposition \ref{prop:PatchingBitorsor}) by taking the last arrow to be $(x_i)_{i \in \Vscr} \mapsto (x_{r(k)}^{-1}|_{F_k} x_{l(k)}|_{F_k})_{k \in \Escr}$ (this need not be a group homomorphism if $G$ or $H$ is not commutative). Let us now construct the ``connecting map''
        \begin{equation*}
            \delta: \prod_{k \in \Escr} (\Ker \del)(F_k) \to \Bit_F(\Cscr)
        \end{equation*}
    as follows. Given $(g_k)_{k \in \Escr} \in \prod_{k \in \Escr} (\Ker \del)(F_k)$, the left translation by $g_k$ defines an automorphism of the trivial $F_k$-bitorsor $(G_{F_k}, \del_{F_k})$ under $\Cscr$ (see Lemma \ref{lemm:TrivialBitorsor}\ref{lemm:TrivialBitorsor:2}). By Proposition \ref{prop:PatchingBitorsor}, there is a unique (up to isomorphism) $F$-bitorsor $(P,f)$ under $\Cscr$ together with isomorphisms $\nu_i: (P_{F_i}, f_{F_i}) \to (G_{F_i},\del_{F_i})$, for $i \in \Vscr$, such that 
        \begin{equation} \label{eq:MayerVietorisCrossedModule:1}
            g_k \nu_{l(k)}(p) = \nu_{r(k)}(p)
        \end{equation}
    for any $k \in \Escr$ and $p \in P(\bar{F}_k)$. We define $\delta((g_k)_{k \in \Escr}):=[P,f]$.

    Let us proceed to show that $\delta$ is a group homomorphism. Let $(g'_k)_{k \in \Escr} \in \prod_{k \in \Escr} (\Ker \del)(F_k)$ be a second family. From this and the product family $(g_k g'_k)_{k \in \Escr} \in \prod_{k \in \Escr} (\Ker \del)(F_k)$, one constructs respective $F$-bitorsors $(P',f')$ and $(P'',f'')$ under $\Cscr$, together with isomorphisms $\nu_i': (P'_{F_i}, f_{F_i}') \to (G_{F_i},\del_{F_i})$ and $\nu_i'': (P''_{F_i}, f_{F_i}'') \to (G_{F_i},\del_{F_i})$ with the properties that
    \begin{equation} \label{eq:MayerVietorisCrossedModule:2}
        g_k' \nu'_{l(k)}(p') = \nu'_{r(k)}(p') \quad \text{and} \quad g_k g_k' \nu''_{l(k)}(p'') = \nu''_{r(k)}(p'')
    \end{equation}
    for any $k \in \Escr$, $p' \in P'(\bar{F}_k)$, and $p'' \in P''(\bar{F}_k)$. Let $\mu: G \times_F G \to G$ denote the multiplication morphism. For $i \in \Vscr$, define an $F_i$-morphism $\pi_i: P'_{F_i} \times_{F_i} P_{F_i} \to P''_{F_i}$ via the commutative square
        \begin{equation} \label{eq:MayerVietorisCrossedModule:3}
            \xymatrix{
                P'_{F_i} \times_{F_i} P_{F_i} \ar[r]^-{\pi_i} \ar[d]^{\nu_i' \times_{F_i} \nu_i} & P''_{F_i} \ar[d]^{\nu_i''} \\
                G_{F_i} \times_{F_i} G_{F_i} \ar[r]^-{\mu_{F_i}} & G_{F_i}
            }
        \end{equation}
    (this is possible since the vertical arrows in \eqref{eq:MayerVietorisCrossedModule:3} are isomorphisms). We claim that $\pi_{l(k),F_k} = \pi_{r(k),F_k}$ for any $k \in \Escr$. Indeed, to simplify notations, put $l:=l(k)$ and $r:=r(k)$. For any $p' \in P'(\bar{F}_k)$ and $p \in P(\bar{F}_k)$, in view of \eqref{eq:MayerVietorisCrossedModule:1}, \eqref{eq:MayerVietorisCrossedModule:2}, and \eqref{eq:MayerVietorisCrossedModule:3}, one has
        \begin{equation*}
            \nu''_r(\pi_r(p',p)) = \nu'_r(p') \nu_r(p) = g_k g'_k \nu'_l(p') \nu_l (p) = g_kg_k' \nu_l''(\pi_l(p',p)) = \nu''_r(\pi_l(p',p))
        \end{equation*}
    (recall that elements $g_k, g'_k \in \Ker \del$ are central in $G$). Since $\nu''_r$ is an isomorphism, it follows that $\pi_{l,F_k} = \pi_{r,F_k}$ as claimed. Since $P$ is affine, Lemma \ref{lemm:PatchingMorphism} tells us that there is an $F$-morphism $\pi: P' \times_F P \to P''$ such that $\pi_{F_i} = \pi_i$ for any $i \in \Vscr$.

    We wish to show that $[P'',f''] = [P,f][P',f']$ in $\Bit_F(\Cscr)$. To this end, fix a vertex $i \in \Vscr$ together with a field inclusion $\bar{F} \hookrightarrow \bar{F}_i$ extending $F \hookrightarrow F_i$. Let $p' \in P'(\bar{F})$, $p \in P(\bar{F})$, and $g \in G(\bar{F})$. Since $f_{F_i} = \del_{F_i} \circ \nu_i$ by construction, one has
        \begin{equation*}
            f(p)^{-1} \cdot g = \del(\nu_i(p))^{-1} \cdot g = \nu_i(p)^{-1} g \nu_i(p)
        \end{equation*}
    by virtue of \eqref{eq:CrossModule}. It follows from \eqref{eq:MayerVietorisCrossedModule:3} that
        \begin{align*}
            \nu_i''(\pi(p' \ast g^{-1}, p \ast (f(p)^{-1} \cdot g))) & = \nu'_i(p' \ast g^{-1}) \nu_i(p \ast (f(p)^{-1} \cdot g)) \\
            & = \nu'_i(p') g^{-1} \nu_i(p) (f(p)^{-1} \cdot g) \\
            & = \nu'_i(p') \nu_i(p) \\
            & = \nu''_i(\pi(p',p)).
        \end{align*}
    Since $\nu''_i$ is an isomorphism, it follows that
        \begin{equation*}
            \pi(p' \ast g^{-1}, p \ast (f(p)^{-1} \cdot g)) = \pi(p',p).
        \end{equation*}
    Using the notations from Remarks \ref{remk:BitorsorGroupLaw}, this implies that $\pi$ induces a morphism $\varphi: P' \times^G_F P \to P''$. In addition, this is a morphism of right $G$-torsors, because
        \begin{equation*}
            \nu_i''(\pi(p',p \ast g)) = \nu_i'(p') \nu_i(p \ast g) = \nu_i'(p') \nu_i(p)g = \nu_i''(\pi(p',p))g = \nu_i''(\pi(p',p) \ast g)
        \end{equation*}
    by \eqref{eq:MayerVietorisCrossedModule:3}, or $\pi(p',p\ast g) = \pi(p',p) \ast g$ because $\nu_i''$ is an isomorphism. It follows that $\varphi$ is a morphism of right $F$-torsors under $G$, hence, as any morphism of torsors, is an isomorphism. 
    
    Now, one has $f_{F_i} = \del_{F_i} \circ \nu_i$, $f_{F_i}' = \del_{F_i} \circ \nu'_i$, and $f''_{F_i} = \del_{F_i} \circ \nu_i''$. Therefore
        \begin{equation*}
            f''(\pi(p',p)) = \del(\nu''_i(\pi(p',p))) = \del(\nu'_i(p') \nu_i(p)) = f'(p')f(p)
        \end{equation*}
    by virtue of \eqref{eq:MayerVietorisCrossedModule:3}. Following the discussion from Remarks \ref{remk:BitorsorGroupLaw}, one sees that the morphism $f'' \circ \pi: P' \times_F P \to H$ induces a morphism $\tilde{f}: P' \times^G_F P \to H$ with the property that $(P' \times^G_F P ,\tilde{f})$ is an $F$-bitorsor under $\Cscr$ whose class in $\Bit_F(\Cscr)$ is precisely $[P,f][P',f']$. Furthermore, one has $f'' \circ \varphi = \tilde{f}$ (because both sides yield $f'' \circ \pi$ when composed with the canonical epimorphism $P' \times_F P \to P' \times_F^G P$). Equivalently, there holds
        \begin{equation*}
            [P,f][P',f'] = [P' \times^G_F P , \tilde{f}] = [P'',f'']
        \end{equation*}
    in $\Bit_F(\Cscr)$. This shows that $\delta: \prod_{k \in \Escr}(\Ker \del)(F_k) \to \Bit_F(\Cscr)$ is effectively a group homomorphism.

    We now show the exactness of \eqref{eq:MayerVietorisCrossedModule} at the term $\prod_{k \in \Escr}(\Ker \del)(F_k)$. Let $(g'_i)_{i \in \Vscr} \in \prod_{i \in \Vscr}(\Ker \del) (F_i)$. We proceed to show that $\delta((g'_{r(k)})^{-1} g'_{l(r)})$ is the neutral element of $\Bit_F(\Cscr)$, {\em i.e.}, the class $[G,\del]$. By construction, $\delta((g'_{r(k)})^{-1} g'_{l(r)})$ is represented by an $F$-bitorsor $(P,f)$ under $\Cscr$ together with isomorphisms $\nu_i: (P_{F_i}, f_{F_i}) \to (G_{F_i}, \del_{F_i})$, for $i \in \Vscr$, such that 
        \begin{equation} \label{eq:MayerVietorisCrossedModule:4}
            g'_{l(k)} \nu_{l(k)}(p) = g'_{r(k)} \nu_{r(k)}(p)
        \end{equation}
    for any $k \in \Escr$ and $p \in P(\bar{F}_k)$. Now, for each $i \in \Vscr$, the left translation by $g'_i$ is an automorphism of the $F_i$-bitorsor $(G_{F_i},\del_{F_i})$ ({\em cf.} Lemma \ref{lemm:TrivialBitorsor}\ref{lemm:TrivialBitorsor:2}). Therefore, composing it with $\nu_i$ yields an isomorphism $\mu_i: (P_{F_i},f_{F_i}) \to (G_{F_i},\del_{F_i})$ with the property that $\mu_{l(k),F_k} = \mu_{r(k),F_k}$ for any $k \in \Escr$ (as follows from \eqref{eq:MayerVietorisCrossedModule:4}). By Lemma \ref{lemm:PatchingMorphism}, we obtain an isomorphism $\mu: (P,f) \to (G,\del)$ of $F$-bitorsors, so that $[P,f] = [G,\del]$ as claimed.
    
    Conversely, let $(g_k)_{k \in \Escr} \in \prod_{k \in \Escr} (\Ker \del)(F_k)$ such that $\delta((g_k)_{k \in \Escr})$ is the class $[G,\del] \in \Bit_F(\Cscr)$. Then there are automorphisms $\nu_i: (G_{F_i}, \del_{F_i}) \to (G_{F_i},\del_{F_i})$, for $i \in \Vscr$, such that
        \begin{equation} \label{eq:MayerVietorisCrossedModule:5}
            g_k \nu_{l(k)}(g) = \nu_{r(k)}(g)
        \end{equation}
    for any $k \in \Escr$ and $g \in G(\bar{F}_k)$. By Lemma \ref{lemm:TrivialBitorsor}\ref{lemm:TrivialBitorsor:2}, for each $i \in \Vscr$, the automorphism $\nu_i$ is the left translation by an element $g'_i \in (\Ker \del)(F_i)$. Substituting $g = 1_G$ in \eqref{eq:MayerVietorisCrossedModule:5} yields $g_k = g'_{r(k)} (g'_{l(k)})^{-1}$, so that $(g_k)_{k \in \Escr}$ indeed comes from $\prod_{i \in \Vscr} (\Ker \del)(F_i)$. This proves that \eqref{eq:MayerVietorisCrossedModule} is exact at the term $\prod_{k \in \Escr}(\Ker \del)(F_k)$.

    Finally, let us show that \eqref{eq:MayerVietorisCrossedModule} is exact at the term $\Bit_F(\Cscr)$. By construction, if a class $x \in \Bit_F(\Cscr)$ lies in the image of $\delta$, its restriction $x|_{F_i} \in \Bit_{F_i}(\Cscr)$ is the neutral class $[G_{F_i},\del_{F_i}]$. Conversely, let $(P,f)$ be an $F$-bitorsor under $\Cscr$ such that there are isomorphisms $\nu_i: (P_{F_i}, f_{F_i}) \to (G_{F_i},\del_{F_i})$ for $i \in \Vscr$. Then, for $k \in \Escr$, the composite
        \begin{equation*}
            \nu_{r(k),F_k} \circ \nu_{l(k),F_k}^{-1}: (G_{F_k}, \del_{F_k}) \to (G_{F_k}, \del_{F_k}).
        \end{equation*}
    is an automorphism of $F_k$-bitorsors. It is then given by left translation by an element $g_k \in (\Ker \del)(F_k)$ (Lemma \ref{lemm:TrivialBitorsor}\ref{lemm:TrivialBitorsor:2}). Therefore, equality \eqref{eq:MayerVietorisCrossedModule:1} holds true. We conclude that $[P,f]$ is the image by $\delta$ of the family $(g_k)_{k \in \Escr}$. Theorem \ref{thm:MayerVietorisCrossedModule} is finally proved.
\end{proof}

Recall that we have a natural inclusion $\H^1(F,\Ker\del) \subseteq \Hbb^0(F,\Cscr)$, {\em cf.} Remarks \ref{remk:FromTorsorToBitorsor}. The group 
    \begin{equation*}
            \Sha^0(F,\Cscr):=\Ker\left(\Hbb^0(F,\Cscr) \to \prod_{i \in \Vscr}\Hbb^0({F_i},\Cscr) \right), 
        \end{equation*}
which consists of classes of $F$-bitorsors $(P,f)$ under $\Cscr$ that become trivial when restricted to $F_i$ (for all $i \in \Vscr$) is isomorphic to the cokernel of the homomorphism
    \begin{equation*}
        \prod_{i \in \Vscr}(\Ker \del)(F_i) \to \prod_{k \in \Escr}(\Ker \del)(F_k), \quad (g_i)_{i \in \Vscr} \mapsto (g_{r(k)}^{-1}g_{l(k)})_{k \in \Escr},
    \end{equation*}
in view of Theorem \ref{thm:MayerVietorisCrossedModule}. On the other hand, by virtue of the classical Mayer--Vietoris sequence \eqref{eq:SixTermMayerVietorisGroup} for $\Ker \del$, we have that $\Sha^0(F,\Cscr)$ coincides with the subgroup
    \begin{equation*}
        \Sha^1(F,\Ker \del) := \left(\H^1(F,\Ker \del) \to \prod_{i \in \Vscr} \H^1(F_i,\Ker \del)\right).
    \end{equation*}
\iffalse 
The vanishing of the three above groups are hence equivalent and can be restated as follows.

\begin{coro} [Local--global principle for bitorsors]
    Let $\Cscr = [G \xrightarrow{\del} H]$ be a crossed module over $F$, with $G$ and $H$ linear. If patching holds for finite-dimensional vector spaces over $\Fcal$, then the following are equivalent.

    \begin{enumerate}
        \item Local--global principle holds for bitorsors under $\Cscr$ over $\Fcal$, that is, for any $F$-bitorsor $(P,f)$ under $\Cscr$, if $f^{-1}(1_H)(F_i) \neq \varnothing$ for all $i \in \Vscr$, then $f^{-1}(1_H)(F) \neq \varnothing$ ({\em cf.} Lemma \ref{lemm:TrivialBitorsor}\ref{lemm:TrivialBitorsor:1}).

        \item Simultaneous factorizations holds for $\Ker \del$ over $\Fcal$, that is, given $g'_k \in (\Ker \del)(F_k)$ for each $k \in \Escr$, there exists $(g_i)_{i \in \Vscr} \in \prod_{i \in \Vscr}(\Ker \del)(F_i)$ such that $g'_{k} = g_{r(k)}^{-1}g_{l(k)}$ for any $k \in \Escr$.

        \item Local--global principle holds for torsors under $\Ker \del$ over $\Fcal$, that is, for any $F$-torsor $P$ under $\Ker \del$, if $(\Ker \del)(F_i) \neq \varnothing$ for all $i \in \Vscr$, then $(\Ker \del)(F) \neq \varnothing$.
    \end{enumerate}
\end{coro}
\fi

\section{The Mayer--Vietoris sequence} \label{sec:MayerVietoris}

 Let $\Fcal = ((F_i)_{i \in \Vscr}, (F_k)_{k \in \Escr})$ be a factorization inverse system of fields with inverse limit a field $F$. Any short complex $\Tscr = [T_1 \xrightarrow{\del} T_2]$ of $F$-tori can be regarded as a crossed module over $F$ (see the beginning of subsection \ref{subsec:CrossedModuleH0}) with the trivial action of $T_2$ on $T_1$. The hypercohomology groups $\Hbb^i(F,T_1 \xrightarrow{\del} T_2)$, for $i = -1, 0$, coincide with the usual Galois hypercohomology groups in the commutative setting. By Theorem \ref{thm:MayerVietorisCrossedModule}, we have a six-term exact sequence
    \begin{equation} \label{eq:SixTermMayerVietorisComplexTori}
        \xymatrix{
            1 \ar[r] & \Hbb^{-1}(F,\Tscr) \ar[r] & \prod_{i \in \Vscr} \Hbb^{-1}(F_i,\Tscr)  \ar[r] & \prod_{k \in \Escr} \Hbb^{-1}(F_k,\Tscr) \ar[lld]_{\delta} \\
            & \Hbb^0(F,\Tscr) \ar[r] & \prod_{i \in \Vscr} \Hbb^0(F_i,\Tscr) \ar[r] & \prod_{k \in \Escr} \Hbb^0(F_k,\Tscr).
        }
    \end{equation}
of abelian groups. Note that the last arrow is also a group homomorphism by commutativity (see its construction from the proof of Theorem \ref{thm:MayerVietorisCrossedModule}). Our aim in this section is to extend \eqref{eq:SixTermMayerVietorisComplexTori} into a nine-term exact sequence using $\Hbb^1$.

\subsection{The case of quasitrivial tori} \label{sec:MayerVietorisQuasitrivialTori}

First, we start with the case of a one-term complex $Q[1]$, where $Q$ is a quasitrivial $F$-torus. Since $\H^1(F_k,Q) = 0$ for any $k \in \Escr$, the result we want is the following.

\begin{thm} \label{thm:PatchingH2QuasiTrivialTorus}
    If patching for finite-dimensional vector spaces holds over $\Fcal$, then we have an exact sequence
        \begin{equation*}
            0 \to \H^2(F,Q) \to \prod_{i \in \Vscr} \H^2(F_i,Q) \to \prod_{k \in \Escr} \H^2(F_k,Q)
        \end{equation*}
    of abelian groups, the last arrow being given by $(\eta_i)_{i \in \Vscr} \mapsto (\eta_{l(k)}|_{F_k} - \eta_{r(k)}|_{F_k})_{k \in \Escr}$.
\end{thm}
\begin{proof}
    Let $\eta \in \H^2(F,Q)$ such that $\eta|_{F_i} = 0$ in $\H^2(F_i,Q)$ for any $i \in \Vscr$. Following the construction of Demarche--Lucchini Arteche \cite[Corollaire 3.3]{DLA2019Reduction}, there exists an integer $n$ and a homogeneous space $X$ of $\SL_{n,F}$ whose ``Springer gerbe'' is represented by the class $\eta \in \H^2(F,Q)$. For any overfield $E/F$, the nonvanishing of $\eta|_E \in \H^2(E,Q)$ is the obstruction to the existence of an $E$-torsor under $\SL_{n}$ dominating $X$. Since $\H^1(E,\SL_n) = 1$ (as follows from Speiser's version of Hilbert's Theorem 90, {\em cf.} \cite[Chapitre X, Corollaire to Proposition 3]{Serre2004Locaux}), one has $\eta|_E = 0$ if and only if $X(E) \neq \varnothing$.

    For any $i \in \Vscr,$ under the assumption that $\eta|_{F_i} = 0$, one has $X(F_i) \neq \varnothing$. Take any point $x_i \in X(F_i)$ and define the morphism
        \begin{equation*}
            \pi_i: \SL_{n,F_i} \to X_{F_i}, \quad q \mapsto x_{i} \ast q.
        \end{equation*}
    (where $\ast$ denotes the right action of $\SL_{n,F}$ on $X$). For any $k \in \Escr$, the fibre $\pi_{l(k),F_k}^{-1}(x_{r(k)})$ is an $F_k$-torsor under $Q$, therefore has an $F_k$-rational point since $\H^1(F_k,Q) = 0$. In other words, there exists $q_k \in Q(F_k)$ such that $x_{l(k)} \ast q_k = x_{r(k)}$. Furthermore, since $\H^1(F,Q) = 0$, by the Mayer--Vietoris sequence \eqref{eq:SixTermMayerVietorisGroup}, there are elements $q'_i \in Q(F_i)$, for $i \in \Vscr$, such that $q'_{l(k)} (q'_{r(k)})^{-1} = q_k$ for any $k \in \Vscr$. Equivalently, one has $x_{l(k)} \ast q'_{l(k)} = x_{r(k)} \ast q'_{r(k)}$. Since $X$ is quasi-projective, there is an affine open subset $U \subseteq X$ containing the points $x_i \ast q'_i$ ($i \in \Vscr$). By Lemma \ref{lemm:PatchingMorphism}, one has $U(F) \neq \varnothing$, {\em a fortiori} $X(F) \neq \varnothing$, hence $\eta = 0$. This shows that the map $\H^2(F,Q) \to \prod_{i \in \Vscr} \H^2(F_i,Q)$ is injective.

    To show exactness at the term $\prod_{i \in \Vscr} \H^2(F_i,Q)$, we need to investigate the effect of tensoring a factorization inverse system of fields with a finite separable extension of its inverse limit.
    
\begin{lemm} \label{lemm:FiniteExtensionFactorizationSystem}
    Let $E/F$ be a finite separable field extension. Write $E = F[X]/(f)$ for some monic irreducible polynomial $f \in F[X]$. For $i \in \Vscr$ (resp. $k \in \Escr$), let
        \begin{equation*}
            f = \prod_{i' \in \Vscr_i} f_{i'} \quad \text{(resp. } f = \prod_{k' \in \Escr_k} f_{k'} \text)
        \end{equation*}
    be the factorization of $f$ in $F_i$ (resp. $F_k$) into monic irreducible factors, where $\Vscr_i$ and $\Escr_k$ are finite sets of indices. Let $E_{i'}:=F_i[X]/(f_{i'})$ (resp. $E_{k'}:=F_k[X]/(f_{k'})$) for $i' \in \Vscr_i$ (resp. $k' \in \Escr_{k'}$. Consider the graph $\Gamma':=(\Vscr',\Escr')$, where
        \begin{equation*}
            \Vscr':=\bigsqcup_{i \in \Vscr} \Vscr_i \quad \text{and} \quad \Escr':=\bigsqcup_{k \in \Escr} \Escr_k,
        \end{equation*}
    and where $k' \in \Escr_k$ is incident on $i' \in \Vscr_i$ if and only if $k$ is incident on $i$ and $f_{k'}$ is a factor of $f_{i'}$ in $F_k[X]$ ({\em i.e.}, one has a natural inclusion $E_{i'} \hookrightarrow E_{k'}$ extending the inclusion $F_i \hookrightarrow F_k$). Then $\Ecal = E \otimes_F \Fcal:= ((E_{i'})_{i' \in \Vscr'}, (E_{k'})_{k' \in \Escr'})$ is an inverse factorization system of fields with inverse limit $E$. If patching for finite-dimensional vector spaces holds over $\Fcal$, then so does it over $\Ecal$.
\end{lemm}
\begin{proof}
     For $i \in \Vscr$ and $k \in \Escr$, one has
        \begin{equation*}
            E \otimes_F F_i = \prod_{i' \in \Vscr_i} E_{i'} \quad \text{and} \quad E \otimes_F F_k = \prod_{k' \in \Escr_k} E_{k'},
        \end{equation*}
    Applying the exact functor $E \otimes_F -$ to the equalizer diagram
        \begin{equation*}
            0 \to F \to \prod_{i \in \Vscr} F_i \rightrightarrows \prod_{k \in \Escr} F_k
        \end{equation*}
    yields an equalizer diagram
        \begin{equation*}
            0 \to E \to \prod_{i' \in \Vscr'} E_{i'} \rightrightarrows \prod_{k' \in \Escr'} E_{k'}
        \end{equation*}
    in the category of commutative rings, where the double arrows are the inclusions $E_{i'} \hookrightarrow E_{k'}$ whenever $k'$ is incident on $i'$. Thus, $E$ is the inverse limit of $\Ecal$. Note that the reduction graph $\Gamma' = (\Vscr',\Escr')$ is equipped with a natural orientation induced from that of $\Gamma$, in such a way that for any edge $k' \in \Escr_k$ (where $k \in \Escr$), one has $l(k') \in \Vscr_{l(k)}$ and $r(k') \in \Vscr_{r(k)}$. 

    Assume that patching for finite-dimensional vector spaces holds over $\Fcal$. We proceed to show that it is also the case over $\Ecal$. By \cite[Proposition 2.2]{HHK2015LocalGlobal}, this amounts to showing that simultaneous factorization holds for $\GL_n$ over $\Ecal$ for any integer $n \ge 1$. Indeed, let $G:=\R_{E/F}(\GL_{n,E})$. Since $\H^1(F,G) = \H^1(E,\GL_n) = 1$ by the ``nonabelian Shapiro lemma's'' (Eckmann--Faddeev--Shapiro, see {\em e.g.} \cite[Chapter VII, Lemma 29.6]{KMRT19998Book}) and Speiser's version of Hilbert's Theorem 90 \cite[Chapitre X, Proposition 3]{Serre2004Locaux}, the Mayer--Vietoris sequence \eqref{eq:SixTermMayerVietorisGroup} tells us that simultaneous factorization holds for $G$ over $\Fcal$, that is, the map
        \begin{equation*}
            \prod_{i \in \Vscr} G(F_i) \to \prod_{k \in \Escr} G(F_k), \quad (g_i)_{i \in \Vscr} \mapsto (g_{r(k)}^{-1} g_{l(k)})_{k \in \Escr}
        \end{equation*}
    is surjective. By definition of Weil restriction, the map
        \begin{equation*}
            \prod_{i' \in \Vscr'} \GL_{n}(E_{i'}) \to \prod_{k' \in \Escr'} \GL_{n}(E_{k'}), \quad (g_{i'})_{i' \in \Vscr'} \mapsto (g_{r(k')}^{-1} g_{l(k')})_{k' \in \Escr'}
        \end{equation*}
    is surjective. Therefore, patching for finite-dimensional vector spaces holds over $\Ecal$ as claimed.
\end{proof}

    Back to the proof of Theorem \ref{thm:PatchingH2QuasiTrivialTorus}. Let us proceed to show exactness at the term $\prod_{i \in \Vscr} \H^2(F_i,Q)$. We may assume that $Q = \R_{E/F}(\Gbb_{m,E})$ for some finite separable field extension $E/F$. Given a family $(\eta_i)_{i \in \Vscr} \in \prod_{i \in \Vscr} \H^2(F_i,Q)$ such that $\eta_{l(k)}|_{F_k} = \eta_{r(k)}|_{F_k}$ for all $k \in \Escr$, we want to show it comes from $\H^2(F,Q)$.
    
    Let $\Ecal = ((E_{i'})_{i' \in \Vscr'},(E_{k'})_{k' \in \Escr'})$ be the factorization inverse system constructed in Lemma \ref{lemm:FiniteExtensionFactorizationSystem}. By Shapiro's lemma, the given family corresponds to a family $(\eta_{i'})_{i' \in \Vscr'} \in \prod_{i' \in \Vscr'} \H^2(E_{i'},\Gbb_m)$ such that $\eta_{l(k')}|_{E_{k'}} = \eta_{r(k')}|_{E_{k'}}$ for all $k' \in \Escr'$. Since $\Vscr'$ is finite, we may choose an integer $n$ such that each $\eta_{i'}$ comes from an element $x_{i'} \in \H^1(E_{i'},\PGL_n)$. For $k' \in \Escr'$, the ``connecting map'' $\H^1(E_{k'},\PGL_n) \to \H^2(E_{k'},\Gbb_m)$ (see {\em e.g.} \cite[Chapitre I, \S{5.7}]{Serre1994Galois}) is injective because Brauer equivalent central simple algebras of the same degree are isomorphic. It follows that $x_{l(k')}|_{E_{k'}} = x_{r(k')}|_{E_{k'}}$. Since patching for finite-dimensional vector spaces holds over $\Fcal$, it also holds over $\Ecal$ by virtue of Lemma \ref{lemm:FiniteExtensionFactorizationSystem}. Therefore, by the Mayer--Vietoris sequence  \eqref{eq:SixTermMayerVietorisGroup} for $\PGL_n$, there is $x \in \H^1(E,\PGL_n)$ such that $x|_{E_{i'}} = x_{i'}$ for all $i' \in \Vscr'$. Its image $\eta' \in \H^2(E,\Gbb_m)$ satisfies $\eta'|_{E_{i'}} = \eta_{i'}$ for all $i' \in \Vscr'$. Again, by Shapiro's lemma, $\eta'$ corresponds to an element $\eta \in \H^2(F,Q)$ with the property that $\eta|_{F_i} = \eta_i$ for all $i \in \Vscr$. The theorem is proved.
\end{proof}

Next, we show that ``patching for $\Hbb^1$'' holds for short complexes of tori. The proof is by d\'evissage. We first need the following homological algebra lemma.

\begin{lemm} \label{lemm:PullbackPushout}
    Consider a commutative square
        \begin{equation} \label{eq:PullbackPushout}
            \xymatrix{
                A \ar[r]^{\del} \ar[d]^{f} & B \ar[d]^{g} \\
                A' \ar[r]^{\del'} & B'
            }
        \end{equation}
    in an abelian category. Assume either of the following.
    \begin{enumerate}
        \item \label{lemm:PullbackPushout:1} $f$ is a monomorphism and \eqref{eq:PullbackPushout} is cocartesian (a pushout diagram).
        \item \label{lemm:PullbackPushout:2} $g$ is an epimorphism and \eqref{eq:PullbackPushout} is cartesian (a pullback diagram).
    \end{enumerate}
    Then, the vertical arrows yield a quasi-isomorphism $[A \xrightarrow{\del} B] \cong [A' \xrightarrow{\del'} B']$ of complexes.
\end{lemm}
\begin{proof}
    \ref{lemm:PullbackPushout:1} is obtained from \ref{lemm:PullbackPushout:2} by working in the opposite category. To prove \ref{lemm:PullbackPushout:2}, by virtue of the Freyd--Mitchell embedding theorem \cite[Theorem 4.4]{Mitchell1964Embedding}, it suffices to work in the category of modules over some ring. Since \eqref{eq:PullbackPushout} is a pullback diagram, we may assume that
        \begin{equation*}
            A = \{(a', b) \in A' \oplus B: \del'(a') = g(b)\},
        \end{equation*}
    and that $f$ and $\del$ are induced by the natural projections onto the two respective factors.

    If $(a',b) \in (\Ker \del) \cap (\Ker f)$, then $a' = 0$ and $b = 0$. This proves that $f$ induces an injection $\Ker \del \hookrightarrow \Ker \del'$. As for surjectivity, take $a' \in A'$ with $\del(a') = 0$ ($= g(0)$). Then $(a',0) \in A' \oplus B$ lies in $\Ker \del$ and is mapped to $a'$ under $f$. Therefore, $f$ restricts to an isomorphism $\Ker \del \xrightarrow{\cong} \Ker \del'$. It remains to show that $g$ induces an isomorphism $\Coker \del \xrightarrow{\cong} \Coker \del'$. Surjectivity follows from assumption. As for injectivity, let $b \in B$ such that $g(b) = \del'(a')$ for some $a' \in A'$. Then $(a',b) \in A$ and $\del(a',b) = b$. The claim follows.
\end{proof}

Let us also recall some facts about lattices. When $G$ is a finite group, a {\em $G$-lattice} is, by definition, a finitely generated torsionfree abelian group equipped with an action of $G$. A {\em permutation} $G$-lattice is one that has a $\mathbb{Z}$-basis permuted by $G$. By Cartier duality ({\em cf.} \nameref{subsec:Notations}), an $F$-torus $T$ is quasitrivial if and only if the $\Gal_{E/F}$-lattice $\hat{T}$ is permutation for some (and hence, for any) finite Galois extension $E/F$ splitting $T$.

\begin{prop} \label{prop:PatchingH1ComplexTori}
    Let $\Tscr = [T_1 \xrightarrow{\del} T_2]$ be a complex of $F$-tori (concentrated in degree $-1$ and $0$). If patching for finite-dimensional vector spaces holds over $\Fcal$, then we have an exact sequence
        \begin{equation*}
            \Hbb^1(F,\Tscr) \to \prod_{i \in \Vscr} \Hbb^1(F_i,\Tscr) \to \prod_{k \in \Escr} \Hbb^1(F_k,\Tscr)
        \end{equation*}
    of abelian groups, the last arrow being given by $(\eta_i)_{i \in \Vscr} \mapsto (\eta_{l(k)}|_{F_k} - \eta_{r(k)}|_{F_k})_{k \in \Escr}$.
\end{prop}
\begin{proof}
    Since any lattice is a quotient of some permutation lattice, we may embed $T_1$ into some quasitrivial $F$-torus $Q$. Define $T$ via the pushout square
        \begin{equation*}
            \xymatrix{
                T_1 \ar[r]^{\del} \ar[d] & T_2 \ar[d] \\
                Q \ar[r] & T.
            }
        \end{equation*}
    By virtue of Lemma \ref{lemm:PullbackPushout}\ref{lemm:PullbackPushout:1}, we have a distinguished triangle
        \begin{equation*}
            Q \to T \to \Tscr \to Q[1]
        \end{equation*}
    Let $(\eta_i)_{i \in \Vscr} \in \prod_{i \in \Vscr}\Hbb^1(F_i,\Tscr)$ such that $\eta_{l(k)}|_{F_k} = \eta_{r(k)}|_{F_k}$ for all $k \in \Escr$. Its image $(a_i)_{i \in \Vscr}$ in $\prod_{i \in \Vscr} \H^2(F_i,Q)$ then satisfies $a_{l(k)}|_{F_k} = a_{r(k)}|_{F_k}$ for any $k \in \Escr$. By Theorem \ref{thm:PatchingH2QuasiTrivialTorus}, there is an element $a \in \H^2(F,Q)$ such that $a|_{F_i} = a_i$ for any $i \in \Vscr$. Take a finite separable extension $E/F$ such that $a|_E = 0$, and let $Q':=\R_{E/F}(Q_E)$ (which is again a quasitrival $F$-torus). Let $j: Q \hookrightarrow Q'$ denote the canonical inclusion (which corresponds to $\id_{Q_E}$ under the bijection $\Mor_F(Q,\R_{E/F}(Q'_E)) \cong \Mor_E(Q_E,Q_E)$). By Shapiro's lemma and \cite[Proposition 1.6.5]{NSW2008Cohomology}, one has $j_\ast a = 0$ in $\H^2(F,Q')$. Let $T'$ be the $F$-torus defined via the pushout square
        \begin{equation*}
            \xymatrix{
                Q \ar[r] \ar[d]^{j} & T \ar[d] \\
                Q' \ar[r] & T'.
            }
        \end{equation*}
    Again, we have a distinguished triangle $Q' \to T' \to \Tscr \to Q'[1]$, by Lemma \ref{lemm:PullbackPushout}\ref{lemm:PullbackPushout:1}. Furthermore, for any $i \in \Vscr$, the image of $\eta_i$ in $\H^2(F_i,Q')$ is $j_\ast a_i = 0$ by construction. Therefore, $\eta_i$ comes from some element $x_i \in \H^1(F_i,T')$. For any $k \in \Escr$, since $\H^1(F_k,Q') = 0$, the map 
        \begin{equation*}
            \H^1(F_k,T') \to \Hbb^1(F_k,\Tscr)
        \end{equation*}
    is injective. Since $\eta_{l(k)}|_{F_k} = \eta_{r(k)}|_{F_k}$, it follows that $x_{l(k)}|_{F_k} = x_{r(k)}|_{F_k}$. By the Mayer--Vietoris sequence \eqref{eq:SixTermMayerVietorisGroup} for $T'$, there exists $x \in \H^1(F,T')$ such that $x|_{F_i} = x_i$ for any $i \in \Vscr$. Its image $\eta$ in $\Hbb^1(F,\Tscr)$ then has the property that $\eta|_{F_i} = \eta_i$ for all $i \in \Vscr$. The proposition is proved.
\end{proof}

\subsection{The nine-term exact sequence}

We are now ready to establish the main result of this section.

\begin{thm} \label{thm:NineTermMayerVietorisComplexTori}
    Let $\Tscr = [T_1 \xrightarrow{\del} T_2]$ be a complex of $F$-tori. If patching for finite-dimensional vector spaces holds over $\Fcal$, then we have a nine-term exact sequence
        \begin{equation} \label{eq:NineTermMayerVietorisComplexTori}
        \xymatrix{
            0 \ar[r] & \Hbb^{-1}(F,\Tscr) \ar[r] & \prod_{i \in \Vscr} \Hbb^{-1}(F_i,\Tscr) \ar[r] & \prod_{k \in \Escr} \Hbb^{-1}(F_k,\Tscr) \ar[lld]_{\delta} \\
            & \Hbb^{0}(F,\Tscr) \ar[r] & \prod_{i \in \Vscr} \Hbb^{0}(F_i,\Tscr) \ar[r] & \prod_{k \in \Escr} \Hbb^{0}(F_k,\Tscr) \ar[lld]_{\delta}\\
            & \Hbb^1(F,\Tscr) \ar[r] & \prod_{i \in \Vscr} \Hbb^1(F_i,\Tscr) \ar[r] & \prod_{k \in \Vscr} \Hbb^1(F_k,\Tscr).
        }
    \end{equation}
\end{thm}
\begin{proof}
    We have seen sequence \eqref{eq:NineTermMayerVietorisComplexTori} for the first six terms ({\em cf.} \eqref{eq:SixTermMayerVietorisComplexTori}) as well as for the last three terms in Proposition \ref{prop:PatchingH1ComplexTori}. Also recall from its proof that there is a distinguished triangle
        \begin{equation} \label{eq:NineTermMayerVietorisComplexTori:1}
            Q \to T \to \Tscr \to Q[1],
        \end{equation}
    where $Q$ and $T$ are $F$-tori and $Q$ is quasitrivial. We now construct the ``connecting homomorphism'' 
    \begin{equation*}
        \delta: \prod_{k \in \Escr}\Hbb^0(F_k,\Tscr) \to \Hbb^1(F,\Tscr).
    \end{equation*}
    Since $\H^1(E,Q) = 0$ for any overfield $E/F$, taking hypercohomology of \eqref{eq:NineTermMayerVietorisComplexTori:1} yields the solid arrows in the commutative diagram 
    \begin{equation} \label{eq:NineTermMayerVietorisComplexTori:2}
        \xymatrix{
            \prod_{i \in \Vscr} Q(F_i) \ar[r] \ar[d] & \prod_{i \in \Vscr} T(F_i) \ar[r] \ar[d] & \prod_{i \in \Vscr}\Hbb^0(F_i, \Tscr) \ar[r] \ar[d] & 0 \\
            \prod_{k \in \Escr} Q(F_k) \ar[r] \ar[d] & \prod_{k \in \Escr} T(F_k) \ar[r] \ar[d]^{\delta} & \prod_{k \in \Escr} \Hbb^0(F_k, \Tscr) \ar[r] \ar@{-->}[d]^{\delta} & 0 \ar[d] \\
            0 \ar[r] & \H^1(F,T) \ar[r] \ar[d] & \Hbb^1(F,\Tscr) \ar[r] \ar[d] & \H^2(F,Q) \ar[d] \\
            0 \ar[r] & \prod_{i \in \Vscr}\H^1(F_i,T) \ar[r] \ar[d] & \prod_{i \in \Vscr}\Hbb^1(F_i,\Tscr) \ar[r] \ar[d] & \prod_{i \in \Vscr}\H^2(F_i,Q) \ar[d] \\
            0 \ar[r] & \prod_{k \in \Escr}\H^1(F_k,T) \ar[r] & \prod_{k \in \Escr}\Hbb^1(F_k,\Tscr) \ar[r] & \prod_{k \in \Escr}\H^2(F_k,Q)
        }
    \end{equation}
    with exact rows. Its columns are exact in view of \eqref{eq:SixTermMayerVietorisGroup}, Theorem \ref{thm:PatchingH2QuasiTrivialTorus}, and Proposition \ref{prop:PatchingH1ComplexTori}. We claim that the surjection $\prod_{k \in \Escr} T(F_k) \to \prod_{k \in \Escr} \Hbb^0(F_k,\Tscr)$ induces an isomorphism
        \begin{equation} \label{eq:NineTermMayerVietorisComplexTori:3}
            \Coker\left(\prod_{i \in \Vscr} T(F_i) \to \prod_{k \in \Escr} T(F_k) \right) \xrightarrow{\cong} \Coker\left(\prod_{i \in \Vscr} \Hbb^0(F_i,\Tscr) \to \prod_{k \in \Escr} \Hbb^0(F_k,\Tscr) \right).
        \end{equation}
    Indeed, let $a \in \prod_{k \in \Escr} T(F_k)$ such that its image $b \in \prod_{k \in \Escr} \Hbb^0(F_k,\Tscr)$ comes from an element $c \in \prod_{i \in \Vscr} \Hbb^0(F_i,\Tscr)$. Then $c$ can be lifted to an element $d \in \prod_{i \in \Vscr} T(F_i)$. Its image $a'$ in $\prod_{k \in \Escr} T(F_k)$ is mapped to $b \in \prod_{k \in \Escr} \Hbb^0(F_k,\Tscr)$, so that the difference $a - a'$ comes from $\prod_{k \in \Escr}Q(F_k)$, hence also from $\prod_{i \in \Vscr}Q(F_i)$, {\em a fortiori} from $\prod_{i \in \Vscr}T(F_i)$. Therefore, $a = (a-a') + a'$ also comes from $\prod_{i \in \Vscr}T(F_i)$. The claim follows.

    Next, let us show that the inclusion $\H^1(F,T) \hookrightarrow \Hbb^1(F,\Tscr)$ induces an isomorphism
        \begin{equation} \label{eq:NineTermMayerVietorisComplexTori:4}
            \Sha^1(F,T) \xrightarrow{\cong} \Sha^1(F, \Tscr).
        \end{equation}
    Indeed, let $a \in \Hbb^1(F, \Tscr)$ be an element whose image in $\prod_{i \in \Vscr}\Hbb^1(F_i, \Tscr)$ is 0. By exactness of the right column, the image of $a$ in $\H^2(F,Q)$ is $0$, therefore $a$ comes from an element $b \in \H^1(F,T)$. By exactness of the fourth row, one has $b \in \Sha^1(F,T)$. 

    The solid arrow $\delta$ from diagram \eqref{eq:NineTermMayerVietorisComplexTori:2} induces an isomorphism
        \begin{equation*}
            \Coker\left(\prod_{i \in \Vscr} T(F_i) \to \prod_{k \in \Escr} T(F_k) \right) \xrightarrow{\cong} \Sha^1(F,T).
        \end{equation*}
    Combining this with \eqref{eq:NineTermMayerVietorisComplexTori:3} and \eqref{eq:NineTermMayerVietorisComplexTori:4} yields an isomorphism
        \begin{equation*}
            \Coker\left(\prod_{i \in \Vscr} \Hbb^0(F_i,\Tscr) \to \prod_{k \in \Escr} \Hbb^0(F_k,\Tscr) \right) \xrightarrow{\cong} \Sha^1(F, \Tscr).
        \end{equation*}
    Therefore, we obtain the dashed arrow in \eqref{eq:NineTermMayerVietorisComplexTori:2}, whose kernel is the image of $\prod_{i \in \Vscr} \Hbb^0(F_i,\Tscr)$ and whose image is $\Sha^1(F,\Tscr)$. The nine-term exact sequence \eqref{eq:NineTermMayerVietorisComplexTori} is hence established.
\end{proof}

The exact sequence \eqref{eq:NineTermMayerVietorisComplexTori} is functorial with respect to the complex $\Tscr$. To see this, we need the notion of {\em coflasque resolution} \cite[\S{0.5}]{CTS1987Flasque}. Recall that a $G$-lattice $C$ (where $G$ is a finite group) is said to be {\em coflasque} if $\H^1(H, C) = 0$ for any subgroup $H \leq G$. This is equivalent to the property that $\Ext^1_{\Zbb[G]}(P, C) = 0$ for any permutation $G$-lattice $P$. We say that an $F$-torus $T$ is coflasque if $\hat{T}$ is a $\Gal_{E/F}$-coflasque lattice for some (and hence, for any) finite Galois extension $E/F$ splitting $T$.

\begin{lemm} \label{lemm:CoflasqueResolution}
    Let $G$ be a finite group.

    \begin{enumerate}
        \item \label{lemm:CoflasqueResolution:1} Any short complex $\Lscr = [L_1 \to L_2]$ of $G$-lattices (where $L_1$ and $L_2$ are placed in respective degrees $-1$ and $0$) admits a {\em coflasque resolution}, {\em i.e.}, is quasi-isomophic to a complex of $G$-lattices of the form $[C \to P]$, where $C$ is coflasque and $P$ is permutation.

        \item \label{lemm:CoflasqueResolution:2} Let $\Lscr$ and $\Lscr'$ be short complexes of $G$-lattices. Let $\Lscr \simeq [C \to P]$ and $\Lscr' \simeq [C' \to P']$ be coflasque resolutions. Then any morphism $f: \Lscr \to \Lscr'$ be a morphism in the derived category of $G$-lattices fits in a commutative diagram of distinguished triangles
        \begin{equation*}
            \xymatrix{
                C \ar[r] \ar[d] & P \ar[r] \ar[d] & \Lscr \ar[r] \ar[d]^{f} & C[1] \ar[d]\\
                C' \ar[r] & P' \ar[r] & \Lscr' \ar[r] & C'[1].
            }
        \end{equation*}  
    \end{enumerate}
     Here, recall that each $G$-lattice is identified to the corresponding one-term complex concentrated in degree $0$.
\end{lemm}
\begin{proof}
\begin{enumerate}
    \item By (0.6.4) from \cite[Lemma 0.6]{CTS1987Flasque}, there is an exact sequence
        \begin{equation*}
            0 \to L_1 \to C_1 \to P_1 \to 0,
        \end{equation*}
    of $G$-lattices, where $P_1$ is permutation and $C_1$ is coflasque. Let $M$ be the $G$-module $C_1 \sqcup_{L_1} L_2$. By Lemma \ref{lemm:PullbackPushout} \ref{lemm:PullbackPushout:1}, $\Lscr$ is quasi-isomorphic to $[C_1 \to M]$, {\em i.e.}, we have a distinguished triangle
        \begin{equation*}
            C_1 \to M \to \Lscr \to C_1[1].
        \end{equation*}
    For any subgroup $H \le G$, since $\H^1(H,C_1) = 0$, the corresponding hypercohomology sequence yields a surjection $M^H \to \Hbb^0(H,\Lscr)$. By (0.6.1) from \cite[Lemma 0.6]{CTS1987Flasque}, there is an exact sequence
        \begin{equation*}
            0 \to C_2 \to P \to M \to 0,
        \end{equation*}
    where $C_2$ (resp. $Q$) is a coflasque (resp. permutation) $G$-lattice. We have $\H^1(H,C_2) = 0$ for any subgroup $H \le G$, so the map $P^H \to M^H$ is surjective. Let $C$ be the $G$-lattice $P \times_M C_1$. By Lemma \ref{lemm:PullbackPushout} \ref{lemm:PullbackPushout:2}, we have quasi-isomorphisms $\Lscr \simeq [C_1 \to M] \simeq [C \to P]$. We claim that $C$ is coflasque. Indeed, taking hypercohomology of the distinguished triangle
        \begin{equation*}
            C \to P \to \Lscr \to C[1]
        \end{equation*}
    yields an exact sequence
        \begin{equation*}
            Q^H \to \Hbb^0(H, \Lscr) \to \H^1(H,C) \to \H^1(H,Q) = 0.
        \end{equation*}
    for any subgroup $H \le G$. The first arrow is surjective because it is the composite
        \begin{equation*}
            Q^H \to M^H \to \Hbb^0(H,\Lscr),
        \end{equation*}
    each arrow being surjective by construction. Therefore, the map $\Hbb^0(H,\Lscr) \to \H^1(H,C)$ is zero. Since it is also surjective, one has $\H^1(H,C) = 0$. Therefore, $C$ is indeed coflasque.

    \item Viewing that $\Hom_{\Dcal(\Zbb[G])}(P,C'[1]) = \Ext^1_{\Zbb[G]}(P,C') = 0$, the composite
        \begin{equation*}
            P \to \Lscr \xrightarrow{f} \Lscr' \to C'[1]
        \end{equation*}
    is zero. The claim follows from \cite[Proposition 1.1.9]{BBD1982Faisceaux}.
\end{enumerate}
\end{proof}

Taking Lemma \ref{lemm:CoflasqueResolution} (and Cartier duality) into account, the functoriality of \eqref{eq:NineTermMayerVietorisComplexTori} becomes clear. Indeed, in the proof of Theorem \ref{thm:NineTermMayerVietorisComplexTori}, we can take a complex $[Q \to T]$ of $F$-tori quasi-isomorphic to $\Tscr$, with $Q$ quasitrivial and $T$ coflasque.
 
If $M$ is a smooth group of multiplicative type over $F$ (that is, the order of $\hat{M}_{\tors}$ is invertible in $F$), then we can find a surjection $T_1 \twoheadrightarrow T_2$ of $F$-tori with kernel $M$ (so that $[T_1 \to T_2] \simeq M[1]$). Theorem \ref{thm:NineTermMayerVietorisComplexTori} now takes the following form, which, in contrast to \cite[Theorem 2.5.1]{HHK2014LocalGlobal}, does not require any ``global domination'' assumption on Galois cohomology of $\Fcal$.

\begin{coro} \label{coro:NineTermMayerVietorisGroupMultiplicativeType}
    Let $M$ be a smooth $F$-group of multiplicative type. If patching for finite-dimensional vector spaces holds over $\Fcal$, then we have a functorial nine-term exact sequence
        \begin{equation*}
        \xymatrix{
            0 \ar[r] & M(F) \ar[r] & \prod_{i \in \Vscr} M(F_i) \ar[r] & \prod_{k \in \Escr} M(F_k) \ar[lld]_{\delta} \\
            & \H^1(F,M) \ar[r] & \prod_{i \in \Vscr} \H^1(F_i,M) \ar[r] & \prod_{k \in \Escr} \H^1(F_k,M) \ar[lld]_{\delta}\\
            & \H^2(F,M) \ar[r] & \prod_{i \in \Vscr} \H^2(F_i,M) \ar[r] & \prod_{k \in \Escr} \H^2(F_k,M).
        }
        \end{equation*}
\end{coro}

\begin{remk}
    One can repeat the proof of \cite[Theorem 2.5.1]{HHK2014LocalGlobal} to extend \eqref{eq:NineTermMayerVietorisComplexTori} into a long exact sequence, provided that the factorization system $\Fcal$ satisfies the global domination condition in \cite[Definition 2.3.3]{HHK2014LocalGlobal}.
\end{remk}

%Long exact sequence with global domination.
\section{Applications} \label{sec:Application}

\subsection{Second Galois cohomology of simply connected groups} \label{subsec:H2}

The notion of band\footnote{{\em lien} in French.} and nonabelian second Galois cohomology has been studied in the last 30 years in order to investigate rational points on homogeneous spaces of algebraic groups. We refer to \cite[\S{1 and 2}]{Borovoi1993H2}, \cite[\S{1}]{FSS1998H2}, and \cite[\S{2.2}]{DLA2019Reduction} for a complete treatment on the subject. {\em Grosso modo}, given a perfect field $F$, an $F$-band is a pair $L = (\bar{G},\kappa)$, where $\bar{G}$ is smooth algebraic $\bar{F}$-group and $\kappa: \Gal_F \to \Aut_{\bar{F}\grp}(\bar{G})/\Inn(\bar{G})$ is a continuous homomorphism satisfying certain continuity and algebraicity conditions (here, $\Aut_{\bar{F}\grp}(\bar{G})$ denotes the group of automorphisms of the $\bar{F}$-group $\bar{G}$ and $\Inn(\bar{G}) \cong \bar{G}(\bar{F})/Z(\bar{G})(\bar{F})$ denotes its subgroup of inner automorphisms). The set $\H^2(F,L)$ of nonabelian second Galois cohomology with coefficients in $L$ has at least three different interpretations: in terms of cocycles, of extensions, and of gerbes \cite[\S{2.2.2}]{DLA2019Reduction}.

The Galois outer action $\kappa$ induces a natural Galois action on the center $Z(\bar{G})$. By Galois descent, this yields an $F$-form $Z(L)$ of $Z(\bar{G})$, and we call this algebraic $F$-group the {\em center of $L$}. If $\H^2(F,L) \neq \varnothing$, then the abelian group $\H^2(F,Z(L))$ acts simply transitively on $\H^2(F,L)$ \cite[\S{1.6} and Lemma 1.7]{Borovoi1993H2} (and this action is functorial in $F$).

Let us consider a factorization inverse system of fields $\Fcal = ((F_i)_{i \in \Vscr}, (F_k)_{k \in \Escr})$ whose inverse limit is a field $F$. Corollary \ref{coro:NineTermMayerVietorisGroupMultiplicativeType} immediately yields the following result on ``patching for nonabelian second Galois cohomology''.

\begin{coro} \label{coro:PatchingNonabelianH2}
    Let $L = (\bar{G},\kappa)$ be an $F$-band such that $Z(L)$ is a smooth group of multiplicative type over $F$. If patching for finite-dimensional vector spaces holds over $\Fcal$ and if $\H^2(F,L) \neq \varnothing$, then one has an equalizer diagram
        \begin{equation*}
            \H^2(F,L) \to \prod_{i \in \Vscr}\H^2(F_i,L) \rightrightarrows \prod_{k \in \Escr}\H^2(F_k,L).
        \end{equation*}
    Moreover, the first arrow is injective if and only if $\Sha^2(F,Z(L)) = 0$.
\end{coro}

In this above statement, if $\bar{G}$ is connected reductive, then the Galois outer action $\kappa$ arises from an $F$-form $G$ \cite[Proposition 3.1]{Borovoi1993H2}. In particular, one has $\H^2(F,L) \neq \varnothing$. Furthermore, the center $Z(L)$ is an $F$-group of multiplicative type. Corollary \ref{coro:PatchingNonabelianH2} now reads

\begin{coro} \label{coro:PatchingNonabelianH2ConnectedReductive}
    Let $L = (\bar{G},\kappa)$ be an $F$-band, with $\bar{G}$ connected reductive. If $Z(L)$ is smooth and if patching for finite-dimensional vector spaces holds over $\Fcal$, then one has an equalizer diagram
        \begin{equation*}
            \H^2(F,L) \to \prod_{i \in \Vscr}\H^2(F_i,L) \rightrightarrows \prod_{k \in \Escr}\H^2(F_k,L).
        \end{equation*}
    The first arrow is injective if and only if $\Sha^2(F,Z(L)) = 0$.
\end{coro}

The main result of this section is the following ``weak local--global principle'' for nonabelian second Galois cohomology of simply connected groups.

\begin{thm} \label{thm:LocalGlobalPrincipleH2SimplyConnected}
    Let $L = (\bar{G},\kappa)$ be an $F$-band, with $\bar{G}$ simply connected semisimple. If $Z(L)$ is smooth (that is, of order invertible in $F$) and if patching for finite-dimensional vector spaces holds over $\Fcal$, then one has an equalizer diagram
        \begin{equation*}
            1 \to \H^2(F,L) \to \prod_{i \in \Vscr}\H^2(F_i,L) \rightrightarrows \prod_{k \in \Escr}\H^2(F_k,L).
        \end{equation*}
\end{thm}
\begin{proof}
    In view of Corollary \ref{coro:PatchingNonabelianH2ConnectedReductive}, it remains to show that $\Sha^2(F,Z(L)) = 0$. Let $G$ be an $F$-form of $\bar{G}$ inducing the Galois outer action $\kappa$ (which exists by \cite[Proposition 3.1]{Borovoi1993H2}). Now, $G$ is isomorphic to a product of $F$-groups of the form $\R_{F'/F}(G')$, where $F'/F$ is a finite separable extension and $G'$ is an absolutely simple simply connected linear algebraic group over $F'$. It is therefore enough to show that $\Sha^2(F,\R_{F'/F}(Z')) = 0$ for any $G'/F'$ as above, where $Z' := Z(G')$. 

    Based on the classification of $G'$, we know that the finite commutative $F'$-group $Z'$ has one of the following form (see {\em e.g.} \cite[pp. 331---332]{PR1994Group}).
    \begin{enumerate}
        \item $Z' = \mu_{n+1}$ (type $\tensor[^{1}]{A}{_n}$).

        \item $Z' = \Ker(\N_{F''/F'}: \R_{F''/F'}(\mu_{n+1}) \to \mu_{n+1})$, where $F''/F'$ is a quadratic extension (type $\tensor[^{2}]{A}{_n}$).

        \item $Z' = \mu_2$ (types $B_n$, $C_n$, and $E_7$).

        \item $Z' = \mu_4$ or $Z' = \mu_2 \times \mu_2$ (type $\tensor[^{1}]{D}{_n}$).

        \item $Z' = \Ker(\N_{F''/F'}: \R_{F''/F'}(\mu_{4}) \to \mu_{4})$ or $Z' = \R_{F''/F'}(\mu_2)$, where $F''/F'$ is a quadratic extension (type $\tensor[^{2}]{D}{_n}$).

        \item $Z' = \Ker(\N_{F''/F'}: \R_{F''/F'}(\mu_{2}) \to \mu_{2})$, where $F''/F'$ is a cubic extension (types $\tensor[^{3}]{D}{_n}$ and $\tensor[^{6}]{D}{_n}$).

        \item $Z' = \mu_3$ (type $\tensor[^{1}]{E}{_6}$).

        \item $Z' = \Ker(\N_{F''/F'}: \R_{F''/F'}(\mu_{3}) \to \mu_{3})$, where $F''/F'$ is a quadratic extension (type $\tensor[^{2}]{E}{_6}$).

        \item $Z' = 1$ (types $E_8$, $F_4$, and $G_2$).
    \end{enumerate}
    
    In any case, $Z'$ is split by a Galois extension $F''/F'$ such that $\Gal_{F''/F'}$ is either $1$, $\mathbb{Z}/2$, $\mathbb{Z}/3$, or the symmetric group on $3$ symbols. Now, by \cite[Lemma 0.6 (0.6.2)]{CTS1987Flasque} and Cartier duality, there is an exact sequence
        \begin{equation*}
            1 \to Z' \to S' \to Q' \to 1,
        \end{equation*}
    where $S'$ is a {\em flasque} $F'$-torus and $Q'$ is a quasitrivial $F'$-torus, both split by $F''$. Since the Sylow subgroups of $\Gal_{F''/F'}$ are cyclic by construction, it follows from a theorem of Endo--Miyata \cite[Theorem 1.5]{EM1975Tori} that $S'$ is a direct factor of a quasitrival $F'$-torus. Applying the exact functor $\R_{F'/F}$ gives an exact sequence
        \begin{equation*}
            1 \to \R_{F'/F}(Z') \to S \to Q \to 1,
        \end{equation*}
    where $S:=\R_{F'/F}(S')$ (resp. $Q:=\R_{F'/F}(Q')$) is a direct factor of (resp. is itself) a quasitrivial $F$-torus. Since $\H^1(F,Q) = 0$, we have an injection $\H^2(F,\R_{F'/F}(Z')) \hookrightarrow \H^2(F,S)$. We have $\Sha^2(F,S) = 0$ by virtue of Theorem \ref{thm:PatchingH2QuasiTrivialTorus}, which implies $\Sha^2(F,\R_{F'/F}(Z')) = 0$, concluding the proof.
\end{proof}

\subsection{Indices of central simple algebras} \label{subsec:Index}

When $A$ is a central simple algebra over a field, denote by $\ind(A)$ its index, that is, the degree (= square root of the dimension) of the central division algebra Brauer-equivalent to $A$. Let $\Fcal = ((F_i)_{i \in \Vscr}, (F_k)_{k \in \Escr})$ be a factorization inverse system of fields whose inverse limit is a field $F$. In this subsection, we rediscover the following well-known local--global principle. 

\begin{thm} \label{thm:Index}
    Let $A$ be a central simple algebra over $F$. Then $\ind(A) = \lcm_{i \in \Vscr} \ind(A_{F_i})$.
\end{thm}

This theorem was proved in \cite[Theorem 5.1]{HHK2009Applications} using the language of {\em generalized Severi--Brauer varieties} and a local--global principle for rational points on homogeneous spaces under rational groups (although the proof in {\em loc. cit.} was presented in the context of semiglobal fields, it actually works for any factorization inverse system). Here we give an alternative proof.

\begin{proof}[Proof of Theorem \ref{thm:Index}]
    It is clear that $\ind(A_{F_i})$ divides $\ind(A)$ for any $i \in \Vscr$. As for the reverse divisiblity, we shall show that for any $n \ge 1$, $\ind(A_{F_i}) | n$ for all $i \in \Vscr$ implies $\ind(A) | n$. Now, for any overfield $E/F$, the class $[A_E] \in \Br(E) = \H^2(E,\Gbb_m)$ lies in the image of the connecting map $\delta_n: \H^1(E,\PGL_n) \to \H^2(E,\Gbb_m)$ (see {\em e.g.} \cite[Chapitre I, \S{5.7}]{Serre1994Galois}) if and only if $A_E$ is Brauer-equivalent to a central simple algebra over $E$ of degree $n$. This latter is equivalent to $\ind(A_E) | n$. Therefore, if  $\ind(A_{F_i}) | n$ for all $i \in \Vscr$, then $[A_{F_i}] = \delta_n(x_i)$ for some $x_i \in \H^1(F_i,\PGL_n)$. For $k \in \Escr$, by injectivity of $\delta_n$, one has $x_{l(k)}|_{F_k} = x_{r(k)}|_{F_k}$. By the Mayer--Vietoris sequence \eqref{eq:SixTermMayerVietorisGroup} for $\PGL_n$, there exists $x \in \H^1(F,\PGL_n)$ such that $x|_{F_i} = x_i$ for all $i \in \Vscr$. Since the map
        \begin{equation*}
            \H^2(F,\Gbb_m) \to \prod_{i \in \Vscr} \H^2(F_i,\Gbb_m)
        \end{equation*}
    is injective by Theorem \ref{thm:PatchingH2QuasiTrivialTorus}, the equality $\delta_n(x)|_{F_i} = \delta_n(x_i) = [A_{F_i}]$ for all $i \in \Vscr$ implies $[A] = \delta_n(x)$, therefore $\ind(A) | n$. The theorem is proved.
\end{proof}

\begin{remk} \label{remk:Index}
    When $F$ is a field and $L = (\bar{G},\kappa)$ is an $F$-band, the set $\H^2(F,L)$ has a subset of {\em neutral} elements, which correspond bijectively to the $F$-forms of $\bar{G}$ inducing the outer Galois action $\kappa$ (these $F$-forms are then inner forms of each other). Assume that $G$ is such an $F$-form and $n(G) \in \H^2(F,L)$ is the corresponding neutral class. Then, any class $\eta \in \H^2(F,L)$ is uniquely written as $\zeta + n(G)$ for some $\zeta \in \H^2(F,Z(L)) = \H^2(F,Z(G))$ (here, ``$+$'' denotes the simply transitive action of $\H^2(F,Z(L))$ on $\H^2(F,L)$). The class $\eta$ is then neutral if and only if $\zeta$ lies in the image of the connecting map $\delta: \H^1(F,G/Z(G)) \to \H^2(F,Z(G))$ ({\em c.f.} \cite[Chapitre I, \S{5.7}]{Serre1994Galois})  induced by the central isogeny
        \begin{equation*}
            1 \to Z(G) \to G \to G/Z(G) \to 1
        \end{equation*}    
    (see {\em e.g.} \cite[Proposition 2.3]{Borovoi1993H2}). Thereom \ref{thm:Index} (and its proof) shows a stronger local--global principle for neutrality in the nonabelian $\H^2$ Galois cohomology sets of absolutely simple simply connected groups of type $^{1} A_{n-1}$.
\end{remk}

{\bf Funding.} This work was supported by the postdoctoral program of the Fondation Sciences Math\'ematiques de Paris.

{\bf Acknowledgements.} The author appreciates the enlightening discussions with Ramdorai Sujatha on the subject. He is grateful to the anonymous referees for their valuable comments and suggestions. He also thanks the Institut de Math\'ematiques de Jussieu-Paris Rive Gauche, CNRS for excellent working conditions. 

\bibliographystyle{alpha}
\bibliography{ref}
\end{document}